\documentclass[3p,11pt,authoryear]{elsarticle}

\usepackage{amsmath}
\usepackage{amssymb}
\usepackage{amsfonts}
\usepackage{acronym}
\usepackage{amsthm}
\usepackage{epigraph}
\usepackage{xcolor}
\usepackage{soul}
\usepackage{booktabs}
\usepackage{graphicx}
\usepackage{algorithm}
\usepackage{algpseudocode}
\usepackage{url}
\usepackage{appendix}
\usepackage{subcaption}

\graphicspath{{figures/}}

\DeclareMathOperator*{\minimize}{minimize}

\newtheorem{example}{Example}

\newtheorem{definition}{Definition}
\newtheorem{proposition}{Proposition}

\newtheorem{theorem}{Theorem}
\newtheorem{note}{Note}
\newtheorem{lemma}{Lemma}

\usepackage{todonotes}

\begin{document}

\begin{frontmatter}

\title{Pairwise comparisons meet optimal transport theory}

\author[aff1]{Matteo Brunelli}
\ead{matteo.brunelli@unitn.it}

\author[aff2]{Silvia Lorenzini}
\ead{silvia.lorenzini@unipg.it}

\address[aff1]{Department of Industrial Engineering, University of Trento, Via Sommarive 9, 38123 Trento, Italy}
\address[aff2]{Department of Economics, University of Perugia, Via Pascoli 20, 06123 Perugia, Italy}

\begin{abstract}
{ In this paper, we reinterpret pairwise comparison theory through the lens of optimal transport. The main contribution is the connection established between pairwise comparison matrices, doubly stochastic scaling, and entropy-regularized transport problems. Using Sinkhorn scaling, we show how consistency, priority derivation, and inconsistency assessment can be studied within a common optimal-transport framework. As methodological consequences of this connection, we obtain a new characterization of consistency, a Sinkhorn-based prioritization method, and an entropy-based inconsistency index. We study these proposals both analytically and numerically and compare them with established methods from the literature.}
\end{abstract}


\end{frontmatter}
\section{Introduction}


Organizations and individuals alike are routinely confronted with decisions that cannot be reduced to a single objective. Whether selecting a supplier, prioritizing projects, allocating resources, or evaluating policy alternatives, decision-makers must typically balance cost, performance, risk, sustainability, and other criteria that may conflict with one another.
Within MCDA, a particularly important role is played by pairwise comparisons, since they make it possible to decompose a complex evaluation task into a collection of simpler judgments involving only two elements at a time. This reduction of cognitive burden is one of the key reasons why pairwise comparison methods remain so attractive in theory and in applications \citep{KouEtAl2016}.

Pairwise comparison matrices have long been used as convenient mathematical structures to collect, organize, and analyze subjective judgments.
Over the years, these matrices have been studied from different mathematical viewpoints. For instance, \citet{LipovetskyConklin2002} interpreted them as contingency tables, \citet{GassRapcsak2004} analyzed them using singular value decomposition, \citet{Barzilai1998} proposed to regard them as points in a linear space, and \citet{KouLin2014} considered them as alignments of column vectors. The existence of multiple interpretations is valuable not only because it provides alternative analytical tools for handling classical issues such as inconsistency estimation and priority derivation, but also because it strengthens the conceptual foundations of the pairwise comparison framework itself. In this sense, each new interpretation may open the way to new methodological developments while, at the same time, offering an independent justification for the use of pairwise comparisons.

Among the classical problems arising in the study of pairwise comparison matrices, two are especially prominent. The first is the derivation of a priority vector, that is, the extraction of a vector of weights capable of synthesizing the information encoded in the judgments. The second is the assessment of inconsistency, namely the quantification of the extent to which a set of judgments departs from perfect coherence. These two issues are deeply related and, together, form the core of much of the literature on pairwise comparisons: { a wide range of priority estimation methods \citep{ChooWedley2004} and inconsistency indices \citep{Brunelli2018} have been developed over the years.} 

Optimal transport theory is the study of optimal transportation and allocation of resources and is highly relevant in { Operational Research}. In its classical discrete form, it concerns the problem of transporting mass from a set of supply nodes to a set of demand nodes in such a way that the total transportation cost is minimized. Although this problem admits a linear programming formulation and can be solved using well-established algorithms such as the transportation simplex method \citep[§9.2]{HillierLieberman2020}, the recent expansion of optimal transport into a wide range of disciplines has considerably increased the demand for faster and more scalable computational procedures. Indeed, opti
mal transport has found important applications in probability theory \citep{Villani2008}, economics \citep{Galichon2016}, pricing \citep{QuEtAl2025}, portfolio optimization \citep{NguyenEtAl2025}, sensitivity analysis \citep{BorgonovoEtAl2025}, and machine learning \citep{Montesuma2024}. In many of these contexts, the relevant instances are large-scale and require algorithms capable of delivering accurate approximations in limited computational time.

A major development in optimal transport is entropic regularization, which replaces the original linear program with a smooth strictly convex optimization problem and leads naturally to Sinkhorn scaling. In particular, the Sinkhorn--Knopp algorithm computes regularized transport solutions by alternating row and column normalizations, while also providing an information-theoretic interpretation of transport plans.

{ In this paper, we use these ideas to reinterpret pairwise comparison theory. 
We view a pairwise comparison matrix as a transformation of the cost matrix 
in an entropy-regularized transport problem, and we use the associated 
Sinkhorn scaling to study consistency, priority derivation, and inconsistency. 
The main contribution of the paper is therefore the connection established 
between pairwise comparisons and optimal transport theory, rather than the 
proposal of an isolated weighting method or inconsistency index. The latter 
arise naturally as methodological consequences of this connection. This 
perspective is relevant to both fields: it provides a new interpretation of 
pairwise comparison matrices and identifies them as a new application domain 
for entropy-regularized optimal transport. We do not claim that the resulting 
methods are universally preferable to existing approaches; rather, they 
provide mathematically sound alternatives grounded in a common 
optimal-transport framework.}

The paper is organized as follows. Section~\ref{sec:pairwise} introduces the basic notions of pairwise comparison theory that are needed in the sequel, with particular attention to priority vectors and inconsistency. Section~\ref{sec:OT} recalls the main concepts from optimal transport theory, including {entropic} regularization and Sinkhorn scaling. Section~\ref{sec:bridge} develops the bridge between the two theories. In particular, we show how the Sinkhorn algorithm can be interpreted as a prioritization method and how the resulting doubly stochastic matrix can be used as a basis for the estimation of inconsistency. Section~\ref{sec:numerical} presents a numerical study aimed at comparing this new perspective with existing interpretation paradigms and established methods from the literature. Finally, Section~\ref{sec:conclusions} concludes the paper with a discussion of the results and possible directions for future research. { All proofs are provided in the Appendix.}

\section{Pairwise comparisons}
\label{sec:pairwise}

In multi-criteria decision analysis (MCDA), it is common practice to decompose preference elicitation problems into smaller and more tractable tasks. A prominent way of doing so is through pairwise assessments, in which only two elements are considered at a time \citep{KouEtAl2016}. This principle is central to several MCDA methods, including Multi-Attribute Value Theory (MAVT) and the Analytic Hierarchy Process (AHP). In MAVT, when an additive value function is assumed, preference information is elicited through trade-off judgments between levels of different criteria. More precisely, experts are asked to determine how an improvement in one criterion can compensate for a deterioration in another, thereby allowing the estimation of scaling constants that reflect the relative importance of the criteria. In AHP, by contrast, experts are asked to directly assess the intensity of importance of one criterion relative to another by means of pairwise comparisons. In both cases, the final goal is the derivation of criteria weights.

We consider a \emph{pairwise comparison matrix} $\mathbf{A}=(a_{ij})_{n \times n}$ whose positive entries are subjective approximations of the ratios between the components of a vector. In accordance with the literature, we assume $a_{ii}=1$ for all $i$, and $a_{ij}=1/a_{ji}$ for all $i,j$. These matrices are well-known because of their extensive use in the Analytic Hierarchy Process \citep{Saaty1977}. We denote by $\mathcal{A}$ the set of all such matrices. The most common interpretation of the entries of a pairwise comparison matrix is as ratios of positive weights $w_1,\ldots,w_n$. Thus, a pairwise comparison matrix has the following structure,
\[
\mathbf{A} = \;
\bordermatrix{ \; & w_1 & w_2 & \cdots & w_n \cr
    w_1 & a_{11} & a_{12} & \cdots & a_{1n} \cr
    w_2 & a_{21} & a_{22} & \cdots & a_{2n} \cr
   \vdots & \vdots & \vdots & \ddots & \vdots \cr 
   w_{n} & a_{n1} & a_{n2} & \cdots & a_{nn}  
}
=
\begin{pmatrix}
    1 & a_{12} & \cdots & a_{1n} \\
    \frac{1}{a_{12}} & 1 & \cdots & a_{2n} \\
    \vdots & \vdots & \ddots & \vdots \\
    \frac{1}{a_{1n}} & \frac{1}{a_{2n}} & \cdots & 1    
\end{pmatrix}\, ,
\]
{where each entry of $\mathbf{A}$ represents a subjective estimation of the ratio between weights, i.e., $a_{ij}\approx w_{i}/w_{j}$.} A pairwise comparison matrix is \emph{consistent} if and only if there exists a \emph{compatible} (priority) vector $\mathbf{w}=(w_1,\ldots,w_n)^{\top}$ such that $a_{ij}={w_i}/{w_j}$ for all $i,j$. Since only ratios between their components matter, two vectors $\mathbf{w}$ and $\mathbf{w}'$ are said to be \emph{equivalent} if there exists $\beta>0$ such that $\mathbf{w}=\beta\mathbf{w}'$. Consistency can also be characterized through a relation on triples of judgments, namely, $a_{ik}=a_{ij}a_{jk}$ for all $i,j,k$. Equivalently, $\mathbf{A}$ is consistent if and only if $\operatorname{rank}(\mathbf{A})=1$. To simplify notation, in the rest of the paper we denote by $\mathcal{A}^{*}$ the set of all consistent pairwise comparison matrices.

\subsection{Priority vector}

Given a pairwise comparison matrix $\mathbf{A}$, the first fundamental task is to derive a suitable priority vector $\mathbf{w}$. The aim is to exploit the information contained in the pairwise ratios in order to obtain a representative estimate of the underlying weights. This operation is trivial if the matrix is consistent, since any column of $\mathbf{A}$ can be taken as a compatible weight vector. By contrast, if $\mathbf{A}$ is inconsistent, there is not a compatible vector and finding a representative one becomes a nontrivial task and has given rise to several different methods and interpretations \citep{Lin2007}.

{ One method is the \emph{eigenvector method}, proposed by \citet{Saaty1977}, who argued} in favor of taking the eigenvector associated with the Perron--Frobenius eigenvalue of $\mathbf{A}$. More precisely, one considers the eigenpair solution of
\begin{equation}
\label{eq:eigenvector}
\mathbf{A}\mathbf{w}=\lambda_{\max}\mathbf{w}
\end{equation}
where $\lambda_{\max}$ is the largest eigenvalue of $\mathbf{A}$.

Another important method is the \emph{geometric mean method}, introduced by \citet{Rabinowitz1976},
\begin{equation}
w_{i}=\left( \prod_{j=1}^{n} a_{ij}\right)^{1/n}
\end{equation}
which corresponds to the analytic solution of a logarithmic least squares problem \citep{CrawfordWilliams1985}.


A further heuristic, sometimes used in practice, is the \emph{normalized columns sum} (NCS) method. First, the columns of $\mathbf{A}$ are normalized so that they sum to 1. Then the resulting matrix $\tilde{\mathbf{A}}=(\tilde{a}_{ij})_{n \times n}$ is considered and the weights are obtained through
\begin{equation}
\label{eq:NCS}
w_{i}=\sum_{j=1}^{n}\tilde{a}_{ij}.
\end{equation}

Since only ratios between weights are relevant, replacing the sum with the average would yield an equivalent vector. This alternative formulation leads naturally to the appealing geometric interpretation illustrated in Figure \ref{fig:interpretation}.

\begin{figure}[ht]
    \centering
    \includegraphics[width=0.35\linewidth]{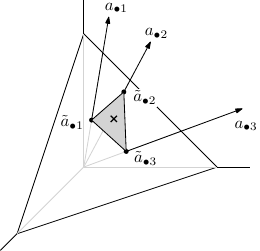}
    \caption{Geometric interpretation of the normalized columns sum method when the sum is replaced by the average. The symbol $\times$ represents the obtained priority vector as the barycenter of the polytope with extreme points $a_{\bullet 1},a_{\bullet 2},a_{\bullet 3}$. Here, $a_{\bullet j}$ represents the $j$th column of the matrix $\mathbf{A}$.}
    \label{fig:interpretation}
\end{figure}

The normalized columns sum method remains, however, rather arbitrary. Indeed, it is not clear why one should normalize columns and then aggregate rows, rather than normalize rows and then aggregate columns. In the consistent case, both procedures lead to equivalent priority vectors; in the inconsistent case---which is by far the most relevant in practice---they generally produce different results. This asymmetry makes it difficult to justify one normalization procedure over the other.

The three priority elicitation methods \eqref{eq:eigenvector}--\eqref{eq:NCS} clearly yield the same priority vector for consistent matrices, but they may produce different weights otherwise. \citet{IshizakaLusti2006} presented a numerical study focused on these three methods, although there exist many others \citep{ChooWedley2004,Lin2007}.

\subsection{Consistency}

Another fundamental issue is the quantification of inconsistency, that is, the assessment of how far the preferences collected in a pairwise comparison matrix depart from consistency. This is usually done by means of \emph{inconsistency indices}, namely, functions $I:\mathcal{A}\rightarrow\mathbb{R}$ such that $I(\mathbf{A})$ quantifies the deviation of $\mathbf{A}$ from consistency. Here we recall only a few representative indices for the sake of comparison.

{ One method was introduced by \citet{Saaty1977}}, who showed that $\lambda_{\max}=n$ if and only if $\mathbf{A}$ is consistent, and is strictly greater otherwise. On the basis of this property, he proposed
\begin{equation}
\label{eq:CI}
CI(\mathbf{A})= \frac{\lambda_{\max}-n}{n-1}
\end{equation}
to estimate the inconsistency of $\mathbf{A}$. { \citet{Saaty1977} proposed to use the average $CI$ of randomly generated matrices as a yardstick against which one can evaluate the value of $CI(\mathbf{A})$.}

\citet{CrawfordWilliams1985} presented an inconsistency index which was later studied by \citet{AguaronMoreno2003} and called the geometric consistency index:
\begin{equation}
GCI(\mathbf{A}) = \frac{2}{(n-1)(n-2)} \sum_{i < j } \log^{2}\left( a_{ij}\frac{w_{j}}{w_{i}} \right),
\end{equation}
where $\mathbf{w}$ is estimated using the geometric mean method.

\citet{PelaezLamata2003} proposed an index that averages the local inconsistencies of triads:
\begin{equation}
PL(\mathbf{A})=\frac{1}{\binom{n}{3}}\sum_{i<j<k} \left( \frac{a_{ij}a_{jk}}{a_{ik}} + \frac{a_{ik}}{a_{ij}a_{jk}} - 2 \right).
\end{equation}
This index is proportional to one previously introduced by \citet{ShiraishiObataDaigo1998}.

\citet{SteinMizzi2007} proposed the harmonic consistency index
\begin{equation}
HCI(\mathbf{A}) = \frac{(HM(\mathbf{A})-n)(n+1)}{n(n-1)},
\end{equation}
where $HM(\mathbf{A})$ is the harmonic mean of the columns of $\mathbf{A}$. In fact, $HM(\mathbf{A}) = n$ if and only if $\mathbf{A}$ is consistent, and it is strictly greater otherwise.

Other indices have been proposed to account for the maximum local inconsistency, thereby departing from the averaging spirit that characterizes most of the literature. The foremost example is the index proposed by \citet{DuszakKoczkodaj1994} and axiomatized by \citet{Csato2018},
\begin{equation}
\label{eq:K}
K(\mathbf{A}) = \max_{i<j<k} \min \left\{ \left| 1- \frac{a_{ik}}{a_{ij}a_{jk}}\right| , \left| 1 - \frac{a_{ij}a_{jk}}{a_{ik}}  \right| \right\} \, .
\end{equation}
.

Overall, a large number of inconsistency indices has been proposed in the literature. A full discussion would go beyond the scope of this paper, and the reader may refer to more specific studies \citep{Brunelli2018,PantEtAl2022}. We only emphasize that inconsistency indices are important not merely as descriptive tools: { their minimization is often used as a criterion for the completion of incomplete matrices and for the revision of judgments that are considered excessively inconsistent \citep{ShiraishiObataDaigo1998,BozokiEtAl2010,Mazurek2023}}.

\section{Optimal transport theory}
\label{sec:OT}

Given $n$ supply nodes and $m$ demand nodes, with supplies $s_{1},\ldots,s_{n} > 0$ and demands $d_1 , \ldots, d_m > 0$, a common problem in operational research consists in the minimization of the transportation cost on the complete bipartite graph with edges $\{ i,j \}$ for all $i=1,\ldots,n$ and $j=1,\ldots,m$. When $\sum_{i}s_{i} = \sum_{j}d_{j}$, the transportation problem is \emph{balanced} and can be written as
\begin{equation}
\tag{OT}
\label{eq:transport}
\begin{aligned}
\minimize_{\mathbf{X} \in \mathbb{R}^{n \times m}} 
& \quad \sum_{i=1}^n \sum_{j=1}^m c_{ij} x_{ij} \\
\text{subject to}
& \quad \sum_{j=1}^m x_{ij} = s_i \quad \forall i = 1,\dots,n, \\
& \quad \sum_{i=1}^n x_{ij} = d_j \quad \forall j = 1,\dots,m, \\
& \quad x_{ij} \ge 0 \quad \forall i,j.
\end{aligned}
\end{equation}
where $c_{ij} \in \mathbb{R}$ is the cost of transporting one unit on the edge $\{ i,j \}$. Such costs can be arranged into a cost matrix $\mathbf{C}=(c_{ij})_{n \times m}$.

\begin{note}
If the values $s_i$ and $d_j$ are normalized so that $\sum_{i}s_{i} = \sum_{j}d_{j} = 1$, then the problem can be interpreted as the optimal transport of one discrete probability distribution into another with minimal cost. This is the most common interpretation, in which the objective function represents the expected transportation cost under a coupling with fixed marginals. Figure \ref{fig:ot} presents a visual representation of the discrete optimal transport problem applied to probability distributions with $n=5$ and $m=6$.
\begin{figure}[h]
    \centering
    \includegraphics[height=6cm]{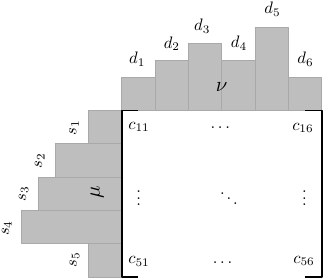}
    \caption{Example of discrete optimal transport with two fixed marginal probability distributions $\mu$ and $\nu$, with $n=5$ and $m=6$.}
 \label{fig:ot}
\end{figure}
Extensions of the transportation problem have also been proposed for non-additive measures \citep{Torra2023} and for nonlinear expectation \citep{LiuEtAl2025,LORENZINI2025}.
\end{note}

The optimization problem \eqref{eq:transport} is a linear program and, according to \citet{Solomon2018}, ``is the most obvious way to make optimal transport work in a discrete context''. However, it typically has sparse solutions and may become computationally demanding at large scale. As recalled by \citet{Solomon2017}, referring to supply and demand nodes, ``modern applications must cope with thousands or millions of these at a time''. Additionally, the problem is non-smooth, which is problematic for optimization and learning.

A balanced transportation problem forms a complete bipartite graph with $|V| = m+n$ nodes and $|E| = mn$ edges. As discussed in \citet{PeyreCuturi2019}, the classical transportation problem can be solved exactly using combinatorial algorithms such as the Goldberg--Tarjan cost-scaling algorithm or the network simplex. The polynomial algorithm by \citet{Orlin1993} has worst-case complexity
\begin{equation}
\label{eq:complexityOT}
\mathcal{O}\left(
|E|\log |V|
\left(|E|+|V|\log |V|\right)
\right).
\end{equation}
The network simplex algorithm, while often faster in practice, has no known polynomial worst-case bound and can exhibit exponential behavior, particularly on dense transportation problems.
For this reason, one may prefer the entropy-regularized optimal transport problem,
\begin{equation}
\tag{E-OT$_{\alpha}$}
\label{eq:transportR}
\begin{aligned}
\minimize_{ \mathbf{X} \in \mathbb{R}_{>}^{n \times m}} 
& \quad \sum_{i=1}^n \sum_{j=1}^m \big( \overbrace{c_{ij} x_{ij}}^{\text{cost}} + \alpha \overbrace{x_{ij} \log x_{ij} }^{\text{regularization}} \big) \\
\text{subject to}
& \quad \sum_{j=1}^m x_{ij} = s_i \quad \forall i = 1,\dots,n, \\
& \quad \sum_{i=1}^n x_{ij} = d_j \quad \forall j = 1,\dots,m,
\end{aligned}
\end{equation}
where $\alpha>0$ is the regularization parameter. The first term measures the transportation cost according to the cost matrix $\mathbf{C}$. The second term corresponds, up to an additive constant when the total transported mass is fixed, to the negative Shannon entropy of the transport plan $\mathbf{X}$ and penalizes highly concentrated couplings.

Now consider the transformation
\begin{equation}
\label{eq:f_alfa}
a_{ij}=f_{\alpha}(c_{ij})=\exp\left(-\frac{c_{ij}}{\alpha}\right).
\end{equation}
Then the objective function in \eqref{eq:transportR} can be rewritten as
\begin{align*}
\sum_{i=1}^n \sum_{j=1}^m \big(c_{ij} x_{ij}+\alpha x_{ij} \log x_{ij}\big)
&= \alpha\sum_{i=1}^n \sum_{j=1}^m x_{ij}\left(\frac{c_{ij}}{\alpha}+\log x_{ij}\right) \\
&= \alpha\sum_{i=1}^n \sum_{j=1}^m x_{ij}\log\frac{x_{ij}}{e^{-c_{ij}/\alpha}} \\
&= \alpha\sum_{i=1}^n \sum_{j=1}^m x_{ij}\log\frac{x_{ij}}{a_{ij}} \\
&= \alpha \, KL(\mathbf{X}||\mathbf{A})
\end{align*}
Thus, once the marginal constraints are fixed, minimizing the regularized objective is equivalent to minimizing the Kullback–Leibler divergence $\mathrm{KL}(\mathbf{X}\|\mathbf{A})$, since the two objectives differ only
by a positive multiplicative factor and an additive constant. The parameter $\alpha > 0$ controls the strength of the regularization: small $\alpha$ yields solutions closer to classical optimal transport, whereas large $\alpha$ produces smoother and more diffuse transport plans.

Under standard assumptions, as $\alpha \to 0$ the entropy-regularized problem converges to the unregularized Kantorovich formulation. Moreover, the regularized problem admits a unique minimizer, which is positive and has the prescribed marginals. This minimizer is diagonally equivalent to the positive kernel $\mathbf{A}$ and can be computed by Sinkhorn scaling.

\begin{theorem}[\citet{Sinkhorn1964}] \label{SinkhornTh}
Given $\mathbf{A} \in \mathbb{R}_{>}^{n \times n}$, there exist diagonal matrices $\mathbf{D}_1$ and $\mathbf{D}_2$ with strictly positive diagonal entries such that $\mathbf{S} = \mathbf{D}_1 \mathbf{A} \mathbf{D}_2$ is doubly stochastic. The matrices $\mathbf{D}_1$ and $\mathbf{D}_2$ are unique up to a common scalar factor.
\end{theorem}

\begin{algorithm}
\caption{Sinkhorn--Knopp Algorithm for Doubly Stochastic Scaling}
\begin{algorithmic}[1]
\Require Positive matrix $\mathbf{A} \in \mathbb{R}_{>}^{n \times n}$, tolerance $\varepsilon>0$, max iterations $K$
\Ensure Doubly stochastic matrix $\mathbf{S} = \operatorname{diag}(r_1,\dots,r_n) \, \mathbf{A} \, \operatorname{diag}(c_1,\dots,c_n)$, scaling vectors $\mathbf{r}, \mathbf{c}$

\State $\mathbf{r} \gets \mathbf{1}_n, \ \mathbf{c} \gets \mathbf{1}_n, \ \mathbf{S} \gets \mathbf{A}$
\For{$k = 1$ to $K$}
        \State $c_j \gets 1 / \sum_i A_{ij} r_i, \ j=1,\dots,n$
        \State $r_i \gets 1 / \sum_j A_{ij} c_j, \ i=1,\dots,n$
    \State $\mathbf{S} \gets \operatorname{diag}(r_1,\dots,r_n) \, \mathbf{A} \, \operatorname{diag}(c_1,\dots,c_n)$
    \If{$\|\mathbf{S}\mathbf{1}_n - \mathbf{1}_n\|_\infty < \varepsilon$ and $\|\mathbf{S}^\top \mathbf{1}_n - \mathbf{1}_n\|_\infty < \varepsilon$} 
        \State \textbf{break}
    \EndIf
\EndFor
\State \Return $\mathbf{S}, \mathbf{r}, \mathbf{c}$
\end{algorithmic}
\end{algorithm}

{ The Sinkhorn factorization can be obtained numerically by means of the Sinkhorn--Knopp algorithm presented in Algorithm 1, which can be interpreted as an alternating normalization of rows and columns}. More generally, it applies to nonnegative matrices with suitable support properties. If $\mathbf{A}$ is a nonnegative square matrix, it is said to have total support if $\mathbf{A}\ne \mathbf{0}$ and every positive entry of $\mathbf{A}$ lies on a positive diagonal. A nonnegative matrix that contains a positive diagonal is said to have support.

\begin{theorem}[\citet{SinkhornKnopp1967}]\label{SinkhornKnopp}
Let $\mathbf{A} \in \mathbb{R}^{n \times n}_{\geq}$. A necessary and sufficient condition for the existence of a doubly stochastic matrix $\mathbf{S}$ of the form $\mathbf{D}_1\mathbf{A}\mathbf{D}_2$, where $\mathbf{D}_1$ and $\mathbf{D}_2$ are diagonal matrices with positive diagonal entries, is that $\mathbf{A}$ has total support. If $\mathbf{S}$ exists, then it is unique. Also, $\mathbf{D}_1$ and $\mathbf{D}_2$ are unique up to a scalar multiple if and only if $\mathbf{A}$ is fully indecomposable.

A necessary and sufficient condition for the iterative process of alternately normalizing the rows and columns of $\mathbf{A}$ to converge to a doubly stochastic limit is that $\mathbf{A}$ has support. If $\mathbf{A}$ has total support, this limit is the matrix $\mathbf{D}_1\mathbf{A}\mathbf{D}_2$. If $\mathbf{A}$ has support that is not total, this limit cannot be of the form $\mathbf{D}_1\mathbf{A}\mathbf{D}_2$.
\end{theorem}

For the purposes of this paper, the most relevant case is the square one with equal row and column marginals. In that setting, the regularized transport plan can be obtained through diagonal scaling of the positive kernel $\mathbf{A}$, and in the special case of unit marginals it becomes a doubly stochastic matrix. A more precise formulation of this connection will be given in Proposition~\ref{prop:SinkhornOTS}.

Adding an entropic regularization term transforms the problem into a
smooth, strictly convex optimization problem that can be solved
efficiently using Sinkhorn iterations \citep{PeyreCuturi2019}. Each
iteration requires $\mathcal{O}(mn)$ arithmetic operations, and the
Sinkhorn scaling can be approximated to any prescribed accuracy in
polynomial time \citep{AltschulerEtAl2017}. This computational
tractability is one of the reasons why entropy-regularized optimal
transport is widely used for large dense problems.

\section{Bridging the two theories}
\label{sec:bridge}

We are ready to examine the parallel between pairwise comparison matrices and the optimal transport theory. First, we establish a connection between the priority vector $\mathbf{w}$ and the Sinkhorn theorem.  We do so by proposing  a new characterization of consistent pairwise comparison matrices obtained using Sinkhorn theorem.




{

\begin{proposition}
\label{prop:main}
Let $\mathbf A\in\mathcal A$, and let $\mathbf S=\mathbf D_1\mathbf A\mathbf D_2$, be its Sinkhorn scaling. Then the following conditions are equivalent:
\begin{enumerate}
\item $\mathbf A\in\mathcal A^*$;
\item $\mathbf S=\frac1n\mathbf 1\mathbf 1^\top$;
\item the vector $\operatorname{diag}(\mathbf{D}_2) = (d^{(2)}_1,\ldots,d^{(2)}_n)$
is compatible with $\mathbf A$, that is, $a_{ij}={d^{(2)}_i}/{d^{(2)}_j}
\; \forall i,j$;
\item the componentwise inverse of
$\operatorname{diag}(\mathbf{D}_1)=(d^{(1)}_1,\ldots,d^{(1)}_n)$
is compatible with $\mathbf A$, that is, $a_{ij}
=
{d^{(1)}_j}/{d^{(1)}_i}
\; \forall i,j$.
\end{enumerate}
\end{proposition}

}


{
Thus, we know that for a consistent pairwise comparison matrix (1.), the core doubly stochastic matrix $\mathbf{S}$ has all entries equal to $1/n$ (2.), one can find a compatible weight vector $\mathbf{w}$ in the diagonal of $\mathbf{D}_2$ (3.) or, equivalently, up to component-wise inversion, in the diagonal of $\mathbf{D}_1$ (4.). All these observations are crucial for the estimations of priorities and inconsistency.
}

Let us underline that $\mathbf{A} \in \mathcal{A}$ has total support, hence it satisfies the condition posed by Theorem \ref{SinkhornKnopp} and we can use the Sinkhorn--Knopp algorithm (Algorithm 1) to compute the Sinkhorn factorization.

\begin{example}
    Consider the consistent pairwise comparison matrix
\begin{equation}
\label{eq:A3x3}
    \mathbf{A}=\begin{pmatrix}
    1 & 6 & 3 \\
    \frac{1}{6} & 1 & \frac{1}{2} \\
    \frac{1}{3} & 2 & 1 
\end{pmatrix} \in \mathcal{A}^{*}
\end{equation}
    The Sinkhorn--Knopp algorithm scales it into
\[
\overbrace{\begin{pmatrix}
    \frac{1}{10} & 0 & 0 \\
    0 & \frac{6}{10} & 0 \\
    0 & 0 & \frac{3}{10} 
\end{pmatrix}}^{\mathbf{D}_1}
\begin{pmatrix}
    1 & 6 & 3 \\
    \frac{1}{6} & 1 & \frac{1}{2} \\
    \frac{1}{3} & 2 & 1 
\end{pmatrix} 
\overbrace{\begin{pmatrix}
    \frac{30}{9} & 0 & 0 \\
    0 & \frac{5}{9} & 0 \\
    0 & 0 & \frac{10}{9} 
\end{pmatrix}}^{\mathbf{D}_2}
=\overbrace{\begin{pmatrix}
    \frac{1}{n} & \frac{1}{n} & \frac{1}{n} \\
    \frac{1}{n} & \frac{1}{n} & \frac{1}{n} \\
    \frac{1}{n} & \frac{1}{n} & \frac{1}{n} 
\end{pmatrix}}^{\mathbf{S}}
\]
where the diagonal of $\mathbf{D}_{2}$ represents a weight vector compatible with the entries of $\mathbf{A}$. An equivalent weight vector appears if we take the component-wise inverse of $\mathbf{D}_{1}$. Conversely, if we still consider $\mathbf{A}$ as in \eqref{eq:A3x3} and we made it inconsistent by setting $a_{13}=2$ and its reciprocal $a_{31}=1/2$, by running Algorithm 1 we obtain
\begin{align*}
\mathbf{D}_1 & =  \begin{pmatrix}
0.116792 & 0 & 0\\
0 & 0.612161 & 0\\
0 & 0 & 0.267386
\end{pmatrix} \\
\mathbf{S} & =
\begin{pmatrix}
0.331313 & 0.379259 & 0.289428 \\
0.289428 & 0.331313 & 0.379259 \\
0.379259 & 0.289428 & 0.331313 
\end{pmatrix} \\
\mathbf{D}_2
& =\begin{pmatrix}
    2.83679 & 0 & 0 \\
    0 & 0.541219 & 0 \\
    0 & 0 & 1.23908 
\end{pmatrix}
\end{align*}

\end{example}

{
 Interestingly, the core $\mathbf{S}$ of a matrix is not affected when the very same matrix is multiplied (through the Hadamard product) with a consistent one. This result is formalized in the following lemma where $\mathbf{S}(\mathbf{A})$ and $\mathbf{S}(\mathbf{B})$ denote the cores of $\mathbf{A}$ and $\mathbf{B}$, respectively.
\begin{lemma}
\label{lemma:invariance-consistent}
Let $\mathbf A\in\mathcal A$ and $\mathbf{A}^{*} \in \mathcal A^*$. Let also $\mathbf B=  \mathbf A \circ \mathbf{A}^{*}$, where $\circ$ denotes the Hadamard product. Then, $
\mathbf S(\mathbf B)=\mathbf S(\mathbf A)$.
\end{lemma}
}

In the consistent case, any column of $\mathbf{A}$ can be taken as the priority vector $\mathbf{w}$ and most of the methods can simply lead to this trivial solution. For example, the characteristic equation of a consistent matrix becomes $\lambda^{\,n-1} (\lambda - n) = 0$, which leads immediately to $\lambda_{\max}=n$. Similarly, under consistency, the Sinkhorn--Knopp algorithm can find the weight vector in one iteration. Namely, in the consistent case, $K=1$ suffices in Algorithm 1.

\begin{proposition}
\label{prop:1step}
    If $\mathbf{A} \in \mathcal{A}^{*}$, then one iteration of the Sinkhorn--Knopp algorithm is sufficient to obtain the weight vector $\mathbf{w}$.
\end{proposition}


Note also that the prioritization method based on the Sinkhorn--Knopp algorithm offers an interpretation for the normalized columns sum method.


\begin{proposition}
\label{prop:first_iteration}
Let $\mathbf{A} \in \mathbb{R}_{>0}^{n \times n}$ and denote the column-normalized matrix  

\[
\tilde{\mathbf{A}} = \mathbf{A} \mathbf{D}_c^{(0)}, \quad 
\mathbf{D}_c^{(0)} = \operatorname{diag}\left(\frac{1}{\sum_i a_{i1}}, \dots, \frac{1}{\sum_i a_{in}}\right).
\]  

Let $w_i^{\mathrm{NCS}} = \sum_j \tilde a_{ij}$ be the \emph{normalized columns sum (NCS) weights}. Then, in the first iteration of the Sinkhorn--Knopp algorithm (columns first), the second column-scaling matrix $\mathbf{D}_2 = \mathbf{D}_c^{(1)}$ can be expressed as

\[
\mathbf{D}_2 = \operatorname{diag}\Bigg( \frac{1}{\sum_{i=1}^{n} \frac{\tilde a_{i1}}{w_i^{\mathrm{NCS}}}}, \ldots, \frac{1}{\sum_{i=1}^{n} \frac{\tilde a_{in}}{w_i^{\mathrm{NCS}}}} \Bigg).
\]  
\end{proposition}


\subsection{Consistency}

At this point, a question could be whether this new view on pairwise comparison matrices represents a fertile ground for a new view on the quantification of inconsistency, too. Given the interpretation of $\mathbf{S}$ as a doubly stochastic matrix, we can consider $\mathbf{S}$ and define a new inconsistency index based on Shannon entropy.

\begin{definition}[Entropy-based inconsistency index] The entropy-based inconsistency of a pairwise comparison matrix $\mathbf{A} \in \mathcal{A}$ is defined as 
    \begin{equation}
I_H(\mathbf{A})  = 1 - \frac{ H(\mathbf{S})}{ n \log_{2}n }
\end{equation}
\end{definition}
with 
\begin{equation}
H(\mathbf{S}) = -\displaystyle \sum_{i=1}^{n} \sum_{j=1}^{n} {s_{ij}}\log_{2}{s_{ij}}
\end{equation}

The second term of $I_H$ corresponds to the normalized entropy of the matrix $\mathbf{S}$ and it is equal to $1$ if and only if $\mathbf{A}\in \mathcal{A}^{*}$ and positive otherwise. Thus, $I_{H}$ is normalized and it takes values in the interval $[0,1[$, which seems to be a property valued in the literature on pairwise comparisons \citep{KoczkodajEtAl2017}.

A natural question is whether $I_{H}$ satisfies some desirable properties. \citet{BrunelliFedrizzi2024} proposed a unifying property but checking whether it holds may be difficult, and it may be easier to consider, individually, the six properties studied by \citet{Brunelli2017}, supplemented
here by invariance under diagonal similarity.

\begin{itemize}
    \item[P1:] (\textit{Null inconsistency}) The inconsistency index is zero if and only if the pairwise comparison matrix $\mathbf{A}$ is consistent: $I(\mathbf{A}) = 0 \Leftrightarrow  \mathbf{A}\in \mathcal{A}^{*}$

    \item[P2:] (\textit{Permutation invariance}) 
    The inconsistency index does not change if the alternatives are permuted. That is, $I(\mathbf{P}\mathbf{A}\mathbf{P}^\top) = I(\mathbf{A})$ for any permutation matrix $\mathbf{P}$.

    \item [P3:]  (\textit{Monotonicity under power transformation})
    Amplifying the preferences of $\mathbf{A}$ preserves or increases their properties, including inconsistency. Formally, if we consider $\mathbf{A}^b = (a_{ij}^b)_{n \times n}$, then 
$b_2>b_1>0$ implies that $I_H(\mathbf A^{b_2})\ge I_H(\mathbf A^{b_1})$.

    \item [P4:]  (\textit{Monotonicity on single comparisons})
    Changing a single comparison of a consistent matrix away from its consistent value increases inconsistency. Consider $\mathbf A\in\mathcal A^*$ with $n\geq 3$ and then let $\mathbf A^{pq}(\delta)$ be obtained
from $\mathbf A$ by replacing $a_{pq}$ with $\delta a_{pq}$ and $a_{qp}$ with $\delta^{-1}a_{qp}$, leaving all other entries unchanged. Then the function $\delta\mapsto I_H(\mathbf A^{pq}(\delta))$ is quasi-convex on $\mathbb R_{>0}$ and has its unique global minimum at $\delta=1$.

    \item [P5:] (\textit{Continuity})
    The inconsistency index $I$ is a continuous function of the entries of $\mathbf{A}$.

    \item [P6:] (\textit{Transposition invariance})
    Inverting preference should not change their consistency; that is $I(\mathbf{A}^{\top}) = I(\mathbf{A})$.
    { \item [P7:] (\textit{Invariance under diagonal similarity}) Rescaling the alternatives by positive factors
    does not change the assessed inconsistency. Formally, for every diagonal matrix
    $\mathbf D = \mathrm{diag}(d_1,\dots,d_n)$, with $d_i>0$, we have $I(\mathbf D \mathbf A\mathbf D^{-1}) = I(\mathbf A)$.}
\end{itemize}

It can be shown that $I_{H}$ satisfies all seven properties. { This represents \textit{prima facie} evidence to support the validity of the inconsistency index $I_H$.}

\begin{proposition}
\label{prop:axioms_consistency}
    The inconsistency index $I_{H}$ satisfies properties P1--{P7}.
\end{proposition}


Triads ($n=3$) play a special role in the analysis of inconsistency, since they are
the smallest pairwise comparison matrices in which inconsistency can arise.
\citet{Csato2019axiomatizations} studied axiomatizations of inconsistency indices on the set
of triads and showed that several natural axioms lead to a unique
inconsistency ranking. In a related direction, \citet{Cavallo2020} analyzed
functional relations and rank correlations among several consistency
indices, emphasizing that different indices may convey essentially the same
information. It is therefore useful to
check how the proposed entropy-based index behaves on triads. The following
result shows that, when \(n=3\), similarly to the other indices considered by \citet{Cavallo2020}, \(I_H\) depends only on the standard triad
deviation from consistency and is monotone with respect to it.
\begin{proposition}
\label{prop:triads}
Let \( \mathbf{A} \in \mathcal A\) with \(n=3\). Then \(I_H\) is a strictly monotone increasing function of $\left| \log {a_{12}a_{23}}{a_{31}} \right|$.
\end{proposition}

\subsection{Priority vectors}

{ Thanks to Proposition \ref{prop:main}, we discussed the potential of the diagonal entries of matrices $\mathbf{D}_1$ and $\mathbf{D}_2$ to be representative of the weight vector $\mathbf{w}$. Given that, in the inconsistent case, they may yield different results and that choosing one over the other would be arbitrary, we propose to average the two results as follows
\begin{definition}
\label{def:prioritization}
    Given $\mathbf{A} \in \mathcal{A}$ and its scaling $\mathbf{S}=\mathbf{D}_1 \mathbf{A} \mathbf{D}_2$ we define the Sinkhorn-induced priority vector as
    \begin{equation}
\label{eq:weights}
w_{i}= \sqrt{{d^{(2)}_i}{{d^{(1)}_{i}}^{-1}}} = \sqrt{\frac{d^{(2)}_i}{d^{(1)}_{i}}}  \quad \forall i
\end{equation}
where $(d_{1}^{(1)},\ldots,d_{n}^{(1)}) = \operatorname{diag}(\mathbf{D}_1)$ and $(d_{1}^{(2)},\ldots,d_{n}^{(2)}) = \operatorname{diag}(\mathbf{D}_2)$. 
\end{definition}

We keep in mind the inverse interpretation of the entries on the diagonal of $\mathbf{D}_1$ and that the obtained vector can be easily normalized so that the sum of its components is equal to 1. We recall that a problem involving arbitrariness arises for the eigenvector method, which considers the right eigenvector, even if similar arguments could justify the use of the left eigenvector \citep{JohnsonEtAl1979,Csato2024}.

A number of studies have proposed desirable properties for prioritization methods. \citet{GolanyKress1993} introduced a set of desiderata, and other axiomatic systems were proposed to characterize the geometric mean method \citep{Barzilai1998,Csato2019,Fichtner1986} and the eigenvector method \citep{Fichtner1986}. If we restrict attention to basic desiderata and regularity properties, the prioritization induced by Sinkhorn scaling satisfies the following ones. Here we will write $\mathbf{w}(\mathbf{A})$ to denote the weight vector obtained from the matrix $\mathbf{A}$.

\begin{itemize}

\item[Q1:] \textit{(Uniqueness \citep{GolanyKress1993})}
The prioritization method returns a unique priority vector up to multiplication by a positive scalar.

\item[Q2:] \textit{(Computability \citep{GolanyKress1993})}
The solution of a prioritization method can be
computed, or approximated to any prescribed accuracy, in polynomial time.

\item[Q3:] \textit{(Correctness \citep{Fichtner1986,Csato2019})}
Whenever $\mathbf A$ is
consistent, it recovers the underlying weight vector that generates
$\mathbf A$, up to multiplication by a positive scalar.

\item[Q4:] \textit{(Permutation equivariance \citep{Barzilai1998,Fichtner1986})}
Relabeling the alternatives leads to the same relabeling of the resulting weights. Formally, for every permutation matrix $\mathbf P$, we have $\mathbf w(\mathbf P\mathbf A\mathbf P^\top)\propto
\mathbf P\mathbf w(\mathbf A)$.

\item[Q5:] \textit{(Inversion with respect to inversion of preferences \citep{GolanyKress1993})}
Replacing $\mathbf A$ with $\mathbf A^\top$ inverts the priority vector.
Formally, $\mathbf w(\mathbf A^\top)\propto {1}/{\mathbf w(\mathbf A)}$,
where the reciprocal is taken component-wise.

\item[Q6:] \textit{(Smoothness \citep{Fichtner1986})}
The mapping from the input matrix $\mathbf A$ to the resulting weight vector is continuous, up to multiplication by a positive scalar.

\item[Q7:] \textit{(Equivariance under diagonal similarity)} Rescaling the alternatives by positive factors
    rescales the weights accordingly. Formally, for every diagonal matrix
    $\mathbf D = \mathrm{diag}(d_1,\dots,d_n)$, with $d_i>0$, we have $\mathbf w(\mathbf D \mathbf A \mathbf D^{-1}) \propto \mathbf D\,\mathbf w(\mathbf A)$.
\end{itemize}

\begin{proposition}
\label{prop:priority-properties}
The prioritization method \eqref{eq:weights} satisfies properties Q1--Q7.
\end{proposition}


}

It is well-known that, if $n=3$, the eigenvector method and the geometric mean method are equivalent \citep[p. 393]{CrawfordWilliams1985}, and simple counterexamples are sufficient to show that the normalized columns sum method does not share this property. It is possible to show that for $n=3$ the weights obtained from \eqref{eq:weights} are equivalent to those obtained with the eigenvector method and the geometric mean method.



{
\begin{proposition}
\label{prop:3x3}
Let $\mathbf{A}\in\mathcal{A}$ and $n=3$. Then the weight vector obtained from \eqref{eq:weights} is equivalent to the weight vector obtained with the geometric mean method and the eigenvector method.
\end{proposition}
}

{Remarkably, it can be shown that the equivalence with the geometric mean method holds also for $n=4$.

\begin{proposition}
\label{prop:4x4}
Let $\mathbf{A}\in\mathcal{A}$ and $n=4$. Then the weight vector obtained from \eqref{eq:weights} is equivalent to the weight vector obtained with the geometric mean method.
\end{proposition}
}

%
%
%



\subsection{Interpretation of the connection}

We can now proceed and formalize the connection that we established between pairwise comparisons and optimal transport. Given a pairwise comparison matrix $\mathbf{A} \in \mathcal{A}$,  we can obtain its corresponding additive cost matrix $\mathbf{C}=(c_{ij})_{n \times n}$, by means of the inverse of \eqref{eq:f_alfa}, i.e.,
\begin{equation}
\label{eq:f_inverse}
    c_{ij} = f^{-1}_{\alpha}(a_{ij})=- \alpha \log a_{ij}.
\end{equation}
Consider the entropy-regularized optimal transport problem \eqref{eq:transportR}  and call $\mathbf{X}_{\alpha}$ its minimizer.
We can ask whether there is a connection between the solution obtained by applying the Sinkhorn scaling to $\mathbf{A}$ and the solution of the entropy regularized transport problem \eqref{eq:transportR}.

\begin{proposition}
\label{prop:SinkhornOTS}
Let $\mathbf{s} = \mathbf{d} =\mathbf{1}$, $\mathbf{A} \in \mathcal{A}$ and $\alpha > 0$, and define $\mathbf{C} = - \alpha \log \mathbf{A}$. Then the solution to the entropy-regularized optimal transport problem \eqref{eq:transportR} coincides with the doubly stochastic matrix obtained from Sinkhorn scaling. That is, $\mathbf{X}_\alpha = \mathbf{S} = \mathbf{D}_1 \mathbf{A} \mathbf{D}_2$.
\end{proposition}

Notice that the role of $\alpha$ in this correspondence differs from
its usual role when the cost matrix is fixed. Indeed, for a fixed
pairwise comparison matrix $\mathbf{A}$, the associated cost matrix is
defined as $\mathbf{C}=-\alpha\log\mathbf{A}$. Hence,
\[
\exp\left(-\frac{\mathbf{C}}{\alpha}\right)=\mathbf{A},
\]
and changing $\alpha>0$ only multiplies the regularized objective by a
positive scalar. Therefore, the resulting minimizer
$\mathbf{X}_{\alpha}=\mathbf{S}$ is independent of $\alpha$. In this
setting, $\alpha$ serves only to determine the scale of the corresponding
additive cost matrix.


This establishes a direct correspondence between the multiplicative pairwise comparison matrix $\mathbf{A}$, the additive cost matrix $\mathbf{C}$, and the doubly stochastic core $\mathbf{S}$. In particular,
Sinkhorn scaling can be interpreted as solving an entropy-regularized optimal transport problem on $\mathbf{C}$, where the diagonal matrices $\mathbf{D}_1$ and $\mathbf{D}_2$ encode the extracted preferential information, while $\mathbf{S}$ provides a normalized, entropy-regularized representation of the original pairwise comparisons.

In the consistent case, Proposition~\ref{prop:main} and the previous discussion imply that one iteration of the Sinkhorn--Knopp algorithm suffices to yield $\mathbf{S} = ({1}/{n}) \mathbf{1} \mathbf{1}^\top$, and the weight vector is fully contained in $\mathbf{D}_2$. Thus, the entropy-regularized transport interpretation is fully compatible with the multiplicative-to-additive correspondence in consistent PCMs. Figure \ref{fig:scheme} presents a conceptual scheme of the connections between pairwise comparison matrices and optimal transport problem.

\begin{figure}[h]
    \centering
    \includegraphics[width=8cm]{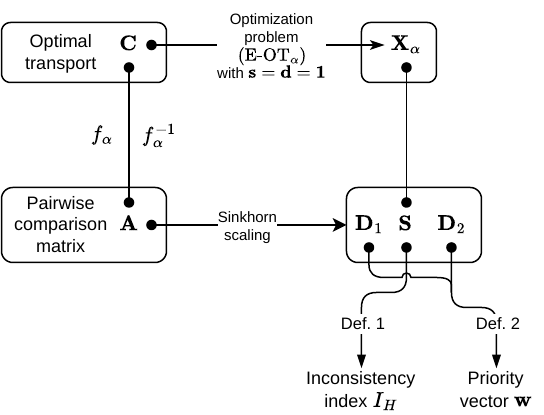}
    \caption{A sketch of the conceptual scheme connecting pairwise comparisons and optimal transport theory.}
    \label{fig:scheme}
\end{figure}

\begin{example}\label{example}
    Consider the following pairwise comparison matrix
    \[
    \mathbf{A}=
    \begin{pmatrix}
  1    & 2 & 2 & 3 \\
  1/2  & 1 &  4  & 1/7 \\
  1/2  & 1/4 & 1 & 1/3 \\
  1/3  &  7  & 3 & 1    
    \end{pmatrix}.
    \]
    If we apply the Sinkhorn--Knopp algorithm to find the Sinkhorn scaling, then
    \[
    \mathbf{S}\approx
    \begin{pmatrix}
 0.2166 & 0.1572 & 0.1068 & 0.5194 \\
 0.2547 & 0.1849 & 0.5023 & 0.0581 \\
 0.4531 & 0.0822 & 0.2233 & 0.2414 \\
 0.0756 & 0.5757 & 0.1676 & 0.1811
 \end{pmatrix}
    \]
and the weight vector found using the prioritization method \eqref{eq:weights} is
    \[
    \mathbf{w}=(0.3985,0.1565,0.0967,0.3483).
    \]
    The doubly stochastic matrix $\mathbf{S}$ can be used to estimate inconsistency of $\mathbf{A}$, leading to $I_{H}(\mathbf{A}) = 0.1457$. The corresponding cost matrix can be obtained by means of \eqref{eq:f_inverse}, in which, for simplicity, we assume $\alpha = 1$,
    \[
    \mathbf{C}=(-\alpha \log a_{ij})_{n \times n} =
    \begin{pmatrix}
 \log 1    & \log 1/2 & \log 1/2 & \log 1/3 \\
  \log 2  & \log 1 &  \log 1/4  & \log 7 \\
  \log 2  &  \log 4 & \log 1 & \log 3 \\
  \log 3  &  \log 1/7  & \log 1/3 & \log 1    
    \end{pmatrix}.
    \]
    If we solve the entropy-regularized optimization problem \eqref{eq:transportR} with $\alpha=1$ and $\mathbf{s}=\mathbf{d}=\mathbf{1}$, we obtain
    \[
    \mathbf{X}_{\alpha}\approx
    \begin{pmatrix}
 0.2166 & 0.1572 & 0.1068 & 0.5194 \\
 0.2547 & 0.1849 & 0.5023 & 0.0581 \\
 0.4531 & 0.0822 & 0.2233 & 0.2414 \\
 0.0756 & 0.5757 & 0.1676 & 0.1811
 \end{pmatrix}.
    \]
    This shows that the Sinkhorn scaling of $\mathbf{A}$ is equivalent, in the unit-marginal square case considered here, to solving the entropy-regularized optimal transport problem \eqref{eq:transportR} associated with $\mathbf{C}$ and $\mathbf{s}=\mathbf{d}=\mathbf{1}$.

\end{example}

We note that $\mathbf{C}=(c_{ij})_{n \times n}$ is an \emph{additive} pairwise comparison matrix with $c_{ii}=0$ for all $i$, and $c_{ij}+c_{ji}=0$ for all $i,j$. Also, if $\mathbf{A}$ is consistent, then $\mathbf{C}$ is additive consistent: $c_{ik}=c_{ij}+c_{jk}$ for all $i,j,k$. This representation of preferences has been studied, among others, by \citet{LavalleFishburn1987} and is employed in some MCDA methods like REMBRANDT \citep{OlsonEtAl1995}.

Let us conclude this section by emphasizing that, under the identification $\mathbf{A}=\exp(-\mathbf{C}/\alpha)$, the Kullback--Leibler divergence measures how far the regularized transport plan $\mathbf{X}$ is from the positive kernel $\mathbf{A}$. In this sense, it quantifies the discrepancy between the original multiplicative structure and its doubly stochastic transport representation. This provides a useful information-theoretic interpretation of the scaling procedure, even if the divergence should not be understood as an independent decision-theoretic loss function.

{Another possible interpretation of the proposed method is that a diagonal similarity
transformation $\mathbf D_1\mathbf A\mathbf D_1^{-1}$ can be interpreted as a reweighting transformation. For both the eigenvector method and the geometric mean method, it is
well known that such a transformation acts equivariantly on the priority vector:
if $\mathbf w$ is the priority vector of $\mathbf A$, then the priority vector of
$\mathbf D_1\mathbf A\mathbf D_1^{-1}$ is $\mathbf D_1\mathbf w$, i.e., the weights are rescaled by the same factors used to reweight the alternatives, so that the change in the priority vector exactly reflects the reweighting transformation (as property Q7 shows). Something analogous happens in
our construction. In fact, by Proposition~\ref{prop:main}, when $\mathbf A$ is consistent,
$\mathbf D_1=\frac1n\mathbf W^{-1}$ and $\mathbf D_2=\mathbf W$, so that
$\mathbf D_2=\frac1n\mathbf D_1^{-1}$: the alternatives are reweighted so that they
become equally important, and $\mathbf S=\mathbf D_1\mathbf A\mathbf D_2$ coincides, up
to the scalar factor $1/n$, with the reweighting transformation
$\mathbf D_1\mathbf A\mathbf D_1^{-1}$.
More generally, $\mathbf D_1$ and $\mathbf D_2$ are independently determined by the
requirement that $\mathbf S$ be doubly stochastic, rather than by a single reweighting
matrix and its inverse; they coincide, up to a scalar factor, with such a matrix and its
inverse if and only if the diagonal entries of $\mathbf S$ are constant, since
$a_{ii}=1$ for all $i$ implies $s_{ii}=d_i^{(1)}d_i^{(2)}$, so that
$\mathbf D_2\propto\mathbf D_1^{-1}$ holds precisely when $s_{ii}$ does not depend on
$i$. This condition always holds when $\mathbf A$ is consistent, and also,
independently of consistency, when $n=3$. 
For $n\geq4$, the condition may still hold for particular inconsistent matrices, but not in
general: this is already visible in Example~\ref{example}, where the diagonal of
$\mathbf S$ is $(0.21664,\,0.184857,\,0.223346,\,0.181125)$, which is not constant, so
that $\mathbf D_2\not\propto\mathbf D_1^{-1}$ in that case. Nevertheless, even when
$\mathbf D_2\not\propto\mathbf D_1^{-1}$, the transformation $\mathbf D_1\mathbf A\mathbf D_2$
strongly resembles the classical reweighting procedure.}

\section{Numerical study}
\label{sec:numerical}

We are ready to study, now by means of numerical experiments, how this new perspective on pairwise comparisons compares to well-known priority elicitation methods and inconsistency indices. We do so with a Monte-Carlo simulation study where we generate a large number $s$ of pairwise comparison matrices of order $n$ on which we apply different analysis techniques. To keep the simulation phase realistic we sampled weights uniformly in the interval $[1,9]$ and we used them to construct consistent pairwise comparison matrices, which are later perturbed with a lognormal noise with $\mu=0$ and $\sigma = 0.43$. This helped obtain preferences that retain some of the initial consistency and are not completely random \citep{Brunelli2018}. During this process, reciprocity is enforced.

First, we analyze prioritization methods with $s=10,\!000$ and $n=5$. For each iteration, we calculate the normalized priority vector using different methods with the addition of a control group where normalized priority vectors are randomly generated { by sampling uniformly from the $(n-1)$-dimensional standard unit simplex}. For each iteration of the simulation, we calculate the distance between pairs of normalized priority vectors. Their average Euclidean and Chebyshev distances over the simulation horizon are reported in the matrix plots in Figure \ref{fig:priorities}.

\begin{figure}[h]
    \centering
\includegraphics[width=0.45\linewidth]{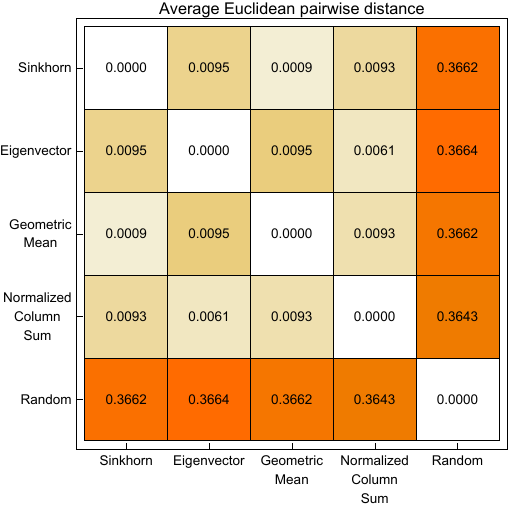}
\hspace{0.7cm}
\includegraphics[width=0.45\linewidth]{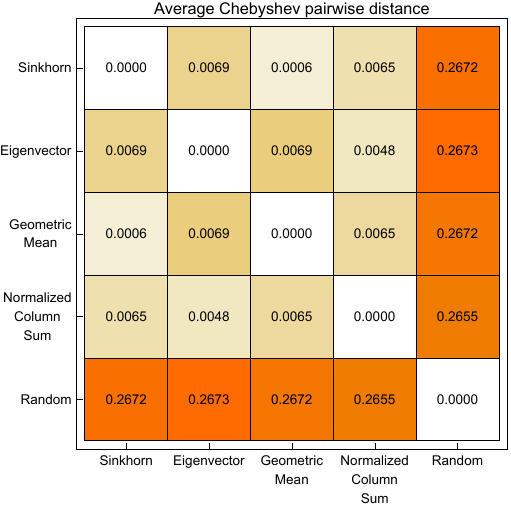}    
    \caption{Average pairwise distances between prioritization methods on randomly generated pairwise comparison matrices with $n=5$.}
    \label{fig:priorities}
\end{figure}

Figure \ref{fig:priorities} shows that, after normalization, the vector obtained as suggested in Definition \ref{def:prioritization} is close to the vectors obtained with the other methods, especially to those obtained via the geometric mean method. At the same time, all four methods form a clearly separated cluster with respect to randomly generated weights. { Tests across different matrix orders $n \geq 5$, which we do not report here, confirmed the results obtained for $n=5$.}

This suggests that prioritization through Sinkhorn scaling is comparable to the established methods and yields meaningful, non-random weights and that it is neither an outlier nor a trivial reformulation, but an alternative prioritization method with its own identity. { Examples show that it leads to substantially different weights when compared to other independent approaches.
\begin{example}
Consider the pairwise comparison matrix
\[
\mathbf{A}=
\begin{pmatrix}
1 & 1 & 3 & 8\\
1 & 1 & 7 & 4\\
1/3 & 1/7 & 1 & 1/2\\
1/8 & 1/4 & 2 & 1
\end{pmatrix}.
\]
If we use the eigenvector method and the Sinkhorn-based prioritization \eqref{eq:weights} we obtain
\[
\mathbf{w}^{\mathrm{EV}}
=
\begin{pmatrix}
0.4239\\
0.4069\\
0.0755\\
0.0937
\end{pmatrix},
\hspace{1cm}
\mathbf{w}^{\mathrm{S}}
=
\begin{pmatrix}
0.4094\\
0.4255\\
0.0726\\
0.0925
\end{pmatrix}
\]
Consequently, the two methods provide different rankings:
\[
\begin{aligned}
\mathbf{w}^{\mathrm{EV}} \quad \Rightarrow \quad
&1 \succ 2 \succ 4 \succ 3,\\
\mathbf{w}^{\mathrm{S}} \quad \Rightarrow  \quad
&2 \succ 1 \succ 4 \succ 3.
\end{aligned}
\]    
\end{example}


Proposition \ref{prop:1step} showed that in the consistent case one iteration of the Sinkhorn algorithm is sufficient: the algorithm can immediately find the priority vector $\mathbf{w}$ satisfying the consistency characterization $a_{ij}={w_i}/{w_j}$ for all $i,j$. More formal results, especially (Q2), are reassuring in terms of worst case complexity. We ran some simulations to understand the convergence of the Sinkhorn--Knopp algorithm to pairwise comparison matrices. We set $\varepsilon = 10^{-10}$ and we counted the number of iterations to reach this tolerance, on a simulation horizon of $s= 10,\!000$. Figure \ref{fig:convergence} shows that a dozen iterations are often sufficient. Moreover, in our implementation, applying the algorithm to a single pairwise comparison matrix required only a few milliseconds.

\begin{figure}[ht]
    \centering
    \includegraphics[width=0.98\linewidth]{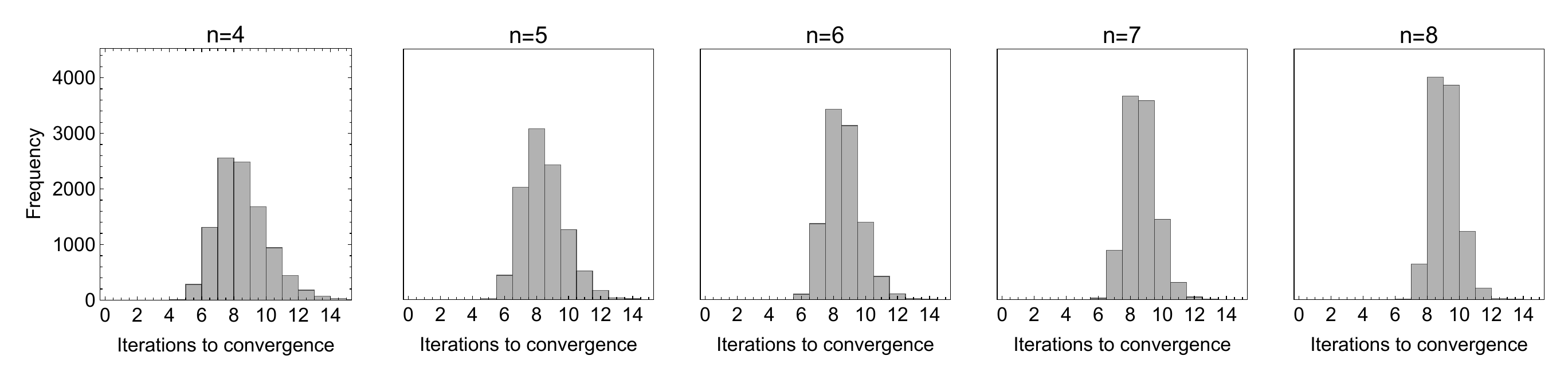}
    \caption{Number of iterations before convergence was obtained, with tolerance $\varepsilon = 10^{-10}$ and $n\in \{ 4,\ldots,8\}$. For each $n$, the number of randomly generated pairwise comparison matrices was $10,\!000$.}
    \label{fig:convergence}
\end{figure}

We used the same simulation strategy to compare the inconsistency index $I_{H}$ with inconsistency indices \eqref{eq:CI}--\eqref{eq:K}. The results are reported in the scatterplot in Figure \ref{fig:scatterplots}.

\begin{figure}[hb]
    \centering
\includegraphics[width=0.98\linewidth]{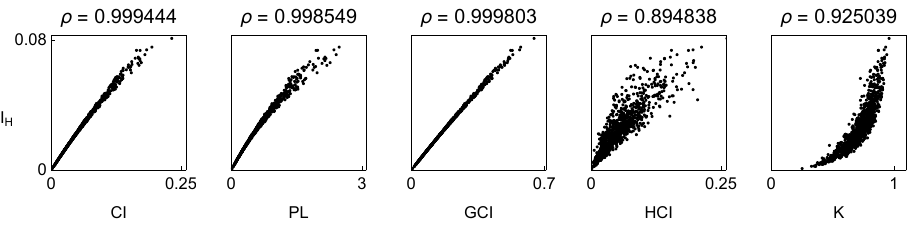}
    \caption{Scatterplots of inconsistency values for randomly generated
reciprocal pairwise comparison matrices with $n=5$. The matrices were
obtained by applying lognormal perturbations to initially consistent
matrices. { For clarity, each plot considers a reduced sample of $1{,}000$ pairwise comparison matrices. Nevertheless, we used the original sample of $10{,}000$ matrices to calculate the Spearman rank correlation coefficients $\rho$ reported on top of each plot.}}
    \label{fig:scatterplots}
\end{figure}

The scatterplots show a behavior that is almost comonotone with the indices $CI$, $GCI$, and $PL$, and substantially different from the other ones. These results suggest that $I_H$ is closely related to the former indices, even if relevant divergences may arise.
\begin{example}
\label{ex:diff_incons}
Consider the following two pairwise comparison matrices whose entries belong to Saaty's scale:
\[
\mathbf{A}=
\begin{pmatrix}
1 & 1/4 & 9 & 1 & 2\\
4 & 1 & 1 & 1/2 & 1/9\\
1/9 & 1 & 1 & 9 & 2\\
1 & 2 & 1/9 & 1 & 8\\
1/2 & 9 & 1/2 & 1/8 & 1
\end{pmatrix}
\hspace{1cm}
\mathbf{A}'=
\begin{pmatrix}
1 & 8 & 1 & 1/9 & 6\\
1/8 & 1 & 1/9 & 9 & 1\\
1 & 9 & 1 & 1/8 & 5\\
9 & 1/9 & 8 & 1 & 1/8\\
1/6 & 1 & 1/5 & 8 & 1
\end{pmatrix}.
\]
We have $CI(\mathbf{A}) = 1.4379$ and $CI(\mathbf{A}') = 2.3955$. Conversely, the assessment is inverted if we consider $I_H(\mathbf{A}) = 0.3954$ and $I_H(\mathbf{A}') = 0.2366$. Given the complexity of the problems underlying the two inconsistency indices, it is difficult to clearly identify the patterns in the preferences
that make one value greater than the other. Nonetheless, we underline a remarkable difference between the two matrices: the largest 3-cycle product in $\mathbf{A}'$ is equal to $a'_{12}a'_{24}a'_{41}=8\cdot9\cdot9=648$ whereas it is equal to $81$ in $\mathbf{A}$. Overall inconsistencies are more uniformly distributed in $\mathbf{A}$ and present peaks in $\mathbf{A}'$.
\end{example}
The results in Example \ref{ex:diff_incons} hint at a substantially different consideration given by inconsistency indices to local inconsistencies, even when the indices appear to be averaging.}

\section{Discussion and conclusions}
\label{sec:conclusions}

In this paper, we proposed a new interpretation of pairwise comparison matrices through the lens of optimal transport theory. By associating a multiplicative pairwise comparison matrix $\mathbf{A}$  with its additive cost matrix $\mathbf{C}=(-\alpha \log a_{ij})_{n\times n}$, we showed that Sinkhorn scaling provides a natural bridge between pairwise comparisons and entropy-regularized optimal transport. This connection is not merely formal: it yields a new interpretation of consistency and leads to two concrete methodological proposals, namely, a prioritization procedure based on Sinkhorn scaling and an entropy-based inconsistency index derived from the associated doubly stochastic matrix.

This perspective also has relevant methodological implications. The equivalence between Sinkhorn scaling and the solution of the entropy-regularized transport problem gives a decision-theoretic interpretation of the scaling procedure, while matrix $\mathbf{S}$ provides a natural basis for measuring inconsistency through Shannon entropy. Moreover, if $n=3$, the prioritization method induced by Sinkhorn scaling is equivalent to the geometric mean method and to the eigenvector method, showing that the proposed approach is compatible with well-established methods while retaining its own interpretation.

The iterative and numerical nature of the proposed prioritization procedure should not be regarded as a drawback. The Sinkhorn--Knopp algorithm is simple, well understood, and computationally efficient. At the same time, the resulting priority vectors appear neither arbitrary nor redundant with respect to standard methods. More broadly, the proposed method occupies an interesting position among weighting procedures: it is exact in the consistent case, yet it also admits an interpretation in terms of entropy-regularized optimization, thus combining a consistency-based and an information-theoretic perspective. { This new interpretation naturally yields a novel inconsistency index and an accompanying prioritization method. While we have verified several of their properties to establish their mathematical soundness, future research should check additional properties of both the inconsistency index \citep{BrunelliFedrizzi2024} and the priority vector \citep{BlanqueroCarrizosa2006,FurtadoJohnson2025}.}

Finally, the present work has been developed in the standard setting of complete pairwise comparison matrices with precise judgments. A natural continuation would be to extend the framework to incomplete and interval-valued pairwise comparison matrices, and more generally to preference models beyond the classical multiplicative setting. For pairwise comparison theory, optimal transport offers a new interpretation of consistency, inconsistency, and prioritization; for optimal transport theory, pairwise comparisons provide a new area of application. We believe that this interaction may stimulate further developments in both directions.

\appendix

\section{Proofs}

\begin{proof}[Proof of Proposition \ref{prop:main}]

$(1) \Rightarrow (2):$ Since $\mathbf{A}$ is consistent, there exists a positive vector $\mathbf{w}$ such that $a_{ij} = {w_i}/{w_j}$ for all $i,j$. Define the diagonal matrices
\[
\mathbf{W} = \operatorname{diag}(w_1,\dots,w_n),
\qquad
\mathbf{W}^{-1} = \operatorname{diag}(w_1^{-1},\dots,w_n^{-1}).
\]

Then $\mathbf{A}$ can be written in matrix form as $\mathbf{A} = \mathbf{W} \, \mathbf{1}\mathbf{1}^\top \, \mathbf{W}^{-1}$ where $\mathbf{1} = (1,\dots,1)^\top \in \mathbb{R}^n$. Define the diagonal matrices
\[
\mathbf{D}_1 = \frac{1}{n} \mathbf{W}^{-1},
\qquad
\mathbf{D}_2 =  \mathbf{W}.
\]
Both $\mathbf{D}_1$ and $\mathbf{D}_2$ have strictly positive diagonal entries. Using $\mathbf{D}_1$ and $\mathbf{D}_2$ as above, we obtain
\[
\begin{aligned}
\mathbf{S}
&= \left( \frac{1}{n} \mathbf{W}^{-1} \right) 
   (\mathbf{W} \mathbf{1}\mathbf{1}^\top \mathbf{W}^{-1})
   \left( \mathbf{W} \right) \\
&= \frac{1}{n} \mathbf{W}^{-1} \mathbf{W} \, \mathbf{1}\mathbf{1}^\top \, \mathbf{W}^{-1} \mathbf{W} \\
&= \frac{1}{n} \mathbf{1}\mathbf{1}^\top.
\end{aligned}
\]
and therefore $(2)$ holds.

$(2)\Rightarrow(3):$ Assume $\mathbf{S}=\frac{1}{n}\mathbf{1}\mathbf{1}^\top$. Since $\mathbf{S}=\mathbf{D}_1\mathbf{A}\mathbf{D}_2$, its entries satisfy $s_{ij}=d^{(1)}_i a_{ij} d^{(2)}_j$. Because $s_{ij}=1/n$ for all $i,j$, we have $a_{ij}=1/(n d^{(1)}_i d^{(2)}_j)$. Setting $i=j$ and using $a_{ii}=1$, we get
\[
1=\frac{1}{n d^{(1)}_i d^{(2)}_i},
\qquad\text{hence}\qquad
d^{(1)}_i=\frac{1}{n d^{(2)}_i}.
\]
Substituting into $a_{ij}=1/(n d^{(1)}_i d^{(2)}_j)$ gives
\[
a_{ij}=\frac{1}{n\left(\frac{1}{n d^{(2)}_i}\right)d^{(2)}_j}=\frac{d^{(2)}_i}{d^{(2)}_j}
\qquad \forall i,j.
\]
and $(3)$ holds.
 
$(3)\Rightarrow(4)$: We know that $a_{ij}=d^{(2)}_i/d^{(2)}_j$ for all $i,j$. Since
$\mathbf S=\mathbf D_1\mathbf A\mathbf D_2$, its entries satisfy
$s_{ij}=d^{(1)}_i a_{ij}d^{(2)}_j$. Hence
$s_{ij}=d^{(1)}_i d^{(2)}_i$ for all $j$. Thus, all entries in the $i$th row
of $\mathbf S$ are equal to $d^{(1)}_i d^{(2)}_i$. Since $\mathbf S$ is
doubly stochastic, the sum of the $i$th row is one, and therefore
$n d^{(1)}_i d^{(2)}_i=1$. It follows that
$d^{(2)}_i=1/(n d^{(1)}_i)$. Substituting this expression into
$a_{ij}=d^{(2)}_i/d^{(2)}_j$, we obtain
\[
a_{ij}
=
\frac{1/(n d^{(1)}_i)}{1/(n d^{(1)}_j)}
=
\frac{d^{(1)}_j}{d^{(1)}_i}
=
\frac{1/d^{(1)}_i}{1/d^{(1)}_j}.
\]
Thus, the component-wise inverse of the vector of diagonal entries of
$\mathbf D_1$ is compatible with $\mathbf A$, and $(4)$ holds.

$(4)\Rightarrow(1)$: We know that $a_{ij}=d^{(1)}_j/d^{(1)}_i$ for all $i,j$. Equivalently,
$a_{ij}=(1/d^{(1)}_i)/(1/d^{(1)}_j)$ for all $i,j$. By setting
$u_i=1/d^{(1)}_i$, we obtain $a_{ij}=u_i/u_j$ for all $i,j$. This is exactly
the definition of consistency of a pairwise comparison matrix. Hence
$\mathbf A\in\mathcal A^*$, and $(1)$ holds.
\end{proof}

\begin{proof}[Proof of Lemma \ref{lemma:invariance-consistent}]
Consider $\mathbf Z=\operatorname{diag}(z_1,\ldots,z_n)$ such that $a^{*}_{ij}=z_i/z_j$. Then we have $\mathbf B=\mathbf A \circ \mathbf{A}^{*}
=
\mathbf Z\mathbf A\mathbf Z^{-1}$ and $\mathbf{A}=\mathbf{Z}^{-1}\mathbf{B}\mathbf{Z}$. From this we can rewrite the Sinkhorn decomposition of $\mathbf{A}$ as
\[
\mathbf S(\mathbf A)
=\mathbf D_1\mathbf A\mathbf D_2
=\mathbf D_1(\mathbf Z^{-1}\mathbf B\mathbf Z)\mathbf D_2
=(\mathbf D_1\mathbf Z^{-1})\,\mathbf B\,(\mathbf Z\mathbf D_2).
\]

Since \(\mathbf D_1\mathbf Z^{-1}\) and \(\mathbf Z\mathbf D_2\) are positive diagonal matrices and \(\mathbf S(\mathbf A)\) is doubly stochastic, this is a valid Sinkhorn scaling of \(\mathbf B\) too. Hence, given the uniqueness, $\mathbf{A}$ and $\mathbf{B}$ share the same core. 
\end{proof}

\begin{proof}[Proof of Proposition \ref{prop:1step}]
Let $\mathbf{A}  \in \mathcal{A}^{*}$ be a consistent pairwise comparison matrix, i.e., there exists $\mathbf{w} \in \mathbb{R}_{>0}^n$ such that $a_{ij} = w_i / w_j$ for all $i,j$. The first step of Sinkhorn scaling normalizes the columns by $c_j = 1 / \sum_i a_{ij}$. For a consistent matrix, $\sum_i a_{ij} = \sum_i w_i / w_j = (\sum_i w_i)/w_j$, so $c_j = w_j / \sum_i w_i$. The column-scaled matrix then has entries $\tilde a_{ij} = a_{ij} c_j = (w_i / w_j) \cdot (w_j / \sum_k w_k) = w_i / \sum_k w_k$, which is constant along each row. The subsequent row scaling computes $r_i = 1 / \sum_j \tilde a_{ij} = (\sum_k w_k)/(n w_i)$, yielding a final scaled matrix $ \mathbf{S} = \operatorname{diag}(\mathbf{r}) \tilde{\mathbf{A}} = (1/n) \mathbf{1} \mathbf{1}^{\top}$, which is exactly doubly stochastic. Therefore, for a consistent pairwise comparison matrix, Sinkhorn scaling converges in a single iteration.
\end{proof}

\begin{proof}[Proof of Proposition \ref{prop:first_iteration}]
After column normalization, the row-scaling vector is  

\[
r_i = \frac{1}{\sum_j \tilde a_{ij}} = \frac{1}{w_i^{\mathrm{NCS}}}, \qquad 
\mathbf{D}_r^{(1)} = \operatorname{diag}(r_1, \dots, r_n).
\]  

The row-scaled matrix is $\mathbf{A}^{(1)} = \mathbf{D}_r^{(1)}\tilde{\mathbf{A}}$. The second column-scaling matrix is defined to normalize the columns of $\mathbf{A}^{(1)}$:

\[
\mathbf{D}_2 = \operatorname{diag}\Bigg( \frac{1}{\sum_i (\mathbf{A}^{(1)})_{i1}}, \dots, \frac{1}{\sum_i (\mathbf{A}^{(1)})_{in}} \Bigg) 
= \operatorname{diag}\Bigg( \frac{1}{\sum_i r_i \tilde a_{i1}}, \dots, \frac{1}{\sum_i r_i \tilde a_{in}} \Bigg).
\]

Substituting $r_i = 1 / w_i^{\mathrm{NCS}}$ gives the desired expression:

\[
\mathbf{D}_2 = \operatorname{diag}\Bigg( \frac{1}{\sum_i \frac{\tilde a_{i1}}{w_i^{\mathrm{NCS}}}}, \dots, \frac{1}{\sum_i \frac{\tilde a_{in}}{w_i^{\mathrm{NCS}}}} \Bigg).
\]

This shows explicitly how the first Sinkhorn iteration transfers the information in the NCS weights into the second column-scaling matrix.  
\end{proof}

\begin{proof}[Proof of Proposition \ref{prop:axioms_consistency}]
{We shall prove each property separately. Only in a few cases the satisfaction of one property is used as an argument to show the satisfaction of another.}

\begin{itemize}
    \item[P1:] The normalized Shannon entropy of $\mathbf{S}$ is equal to one if and only $\mathbf{S}$ is uniform, which, as shown in Proposition \ref{prop:main}, is the case if and only if $\mathbf{A} \in \mathcal{A}^{*}$; therefore, P1 is satisfied.
    \item[P2:] If we consider that $\mathbf{P}^{\top}\mathbf{P}=\mathbf{I}$ then we also obtain
\[
(\mathbf{P} \mathbf{D}_1  \mathbf{P}^{\top})  (\mathbf{P}  \mathbf{A}  \mathbf{P}^{\top})  (\mathbf{P}  \mathbf{D}_2  \mathbf{P}^{\top}) 
= \mathbf{P} \, \mathbf{D}_1  \mathbf{A}  \mathbf{D}_2  \mathbf{P}^{\top} 
= \mathbf{P}  \mathbf{S}  \mathbf{P}^{\top}.
\]
so that we know that row and columns permutations change the order of the entries in $\mathbf{S}$ but not their values. Given that the normalized Shannon entropy is a commutative function, we know that P2 is satisfied.
{ \item[P3:] Let the following represent the linear and the entropic terms of (E-OT$_{\alpha}$)
\[
L(\mathbf X):=\langle \log\mathbf A,\mathbf X\rangle,
\qquad
H(\mathbf X):=-\sum_{i,j}x_{ij}\log x_{ij}.
\]
For the moment we assume $\alpha=1$ but later we'll see that we will not lose in generality. For each $b>0$, the Sinkhorn core of $\mathbf A^{b}$ is characterized by
\[
\mathbf S(\mathbf A^{(b)})
=
\arg\min_{\mathbf X\in\mathcal U_n}
\left\{
-bL(\mathbf X)-H(\mathbf X)
\right\},
\]
where $\mathcal U_n$ is the set of doubly stochastic matrices of order $n$. Let $0<b_1<b_2$ and set
\[
\mathbf X_1=\mathbf S(\mathbf A^{b_1}),
\qquad
\mathbf X_2=\mathbf S(\mathbf A^{b_2}).
\]
By optimality,
\begin{align*}
-b_1L(\mathbf X_1)-H(\mathbf X_1)
& \le
-b_1L(\mathbf X_2)-H(\mathbf X_2),
 \\
-b_2L(\mathbf X_2)-H(\mathbf X_2)
& \le
-b_2L(\mathbf X_1)-H(\mathbf X_1).
\end{align*}
If we add and manipulate them, we reach $(b_2-b_1)\bigl(L(\mathbf X_2)-L(\mathbf X_1)\bigr)\ge 0$,
and therefore $L(\mathbf X_2)\ge L(\mathbf X_1)$.
Using the first optimality inequality, we obtain
\[
H(\mathbf X_1)-H(\mathbf X_2)
\ge
b_1\bigl(L(\mathbf X_2)-L(\mathbf X_1)\bigr)\ge 0.
\]
Thus $H(\mathbf S(\mathbf A^{b_1}))
\ge
H(\mathbf S(\mathbf A^{b_2}))$. Since $I_H(\mathbf A)=1-{H(\mathbf S(\mathbf A))}/{(n\log n)}$, we conclude that $I_H(\mathbf A^{b_1})
\le
I_H(\mathbf A^{b_2})$.

\item[P4:] Since $\mathbf A$ is consistent, there exists $\mathbf w>\mathbf 0$ such that $a_{ij}={w_i}/{w_j}$ for all $i,j$. Let $\mathbf C$ be the consistent matrix generated by $\mathbf w$,
and let $\mathbf E^{pq}(\delta)$ be the reciprocal matrix whose entries are all
equal to one, except $e^{pq}_{pq}(\delta)=\delta$ and 
$e^{pq}_{qp}(\delta)=\delta^{-1}$. Then $\mathbf A^{pq}(\delta)=\mathbf C \circ \mathbf E^{pq}(\delta)$. Since $\mathbf C$ is consistent, Lemma~\ref{lemma:invariance-consistent}
implies $I_H(\mathbf A^{pq}(\delta))
=
I_H(\mathbf E^{pq}(\delta))$. Thus, the consistent component of $\mathbf A^{pq}(\delta)$ can be removed, and it is enough to study the pure perturbation $\mathbf E^{pq}(\delta)$. By permutation invariance, we may assume without loss of generality that
$(p,q)=(1,2)$. Write $\mathbf E(\delta):=\mathbf E^{12}(\delta)$ and let $\mathbf B:=\mathbf E(e)$.
For $\delta\ge 1$, we have $\mathbf E(\delta)=\mathbf B^{(\log\delta)}$, where the power is entrywise. Hence, by monotonicity of $I_H$ under power
transformations (P3), $\delta\mapsto I_H(\mathbf E(\delta))$ is nondecreasing on $[1,\infty)$. For $0<\delta\le 1$, observe that $\mathbf E(\delta)=\mathbf E(1/\delta)^\top$. By transposition invariance of $I_H$ (P6), $I_H(\mathbf E(\delta))
= I_H(\mathbf E(1/\delta))$. Therefore $\delta\mapsto I_H(\mathbf E(\delta))$ is nonincreasing on $(0,1]$.
Thus $\delta\mapsto I_H(\mathbf E(\delta))$ is quasi-convex on $\mathbb R_{>0}$. Finally, $\mathbf E(1)$ is consistent, whereas $\mathbf E(\delta)$ is inconsistent for every $\delta\ne1$. Therefore, thanks to P1, we know that $\delta=1$ is the unique global minimum.}

\item[P5:] To prove P5, let us note that the doubly stochastic matrix $\mathbf{S}$ obtained by
Sinkhorn scaling depends continuously on the entries of $\mathbf{A}$ \citep{Sinkhorn1972}. Since $I_H(\mathbf{A})$ is defined as the normalized Shannon entropy of $\mathbf{S}$, and
Shannon entropy is a continuous function of the matrix entries, it follows that
$I_H$ is continuous with respect to $\mathbf{A}$.

\item[P6:] We check what happens if we transpose $\mathbf{A}$ and we obtain $\mathbf{D}_{2}\mathbf{A}^{\top}\mathbf{D}_{1} = (\mathbf{D}_{1}\mathbf{A}\mathbf{D}_{2})^{\top}
= \mathbf{S}^{\top}$. So, if we transpose $\mathbf{A}$ we obtain $\mathbf{S}^{\top}$, and using the same line of thought used for P2, i.e., the commutativity of the aggregation in $I_{H}$, we know that P6 holds too.
{\item[P7:] Let us define $\mathbf A^* = (d_i/d_j)_{n\times n}$. Since $\mathbf A^*$ is generated by the positive vector
$\mathbf d$, we have $\mathbf A^* \in \mathcal{A}^*$. The entries of $\mathbf B := \mathbf D \mathbf A \mathbf D^{-1}$
satisfy $b_{ij} = (d_i/d_j)a_{ij} = a^*_{ij}a_{ij}$, so $\mathbf B = \mathbf A  \circ \mathbf{A}^*$. By Lemma \ref{lemma:invariance-consistent},
$\mathbf S(\mathbf B) = \mathbf S(\mathbf A)$, and so $I_H(\mathbf B) = I_H(\mathbf A)$.}

\end{itemize}

Each property is satisfied and therefore $I_H$ satisfies P1--{P7}.
\end{proof}

{
\begin{proof}[Proof of Proposition \ref{prop:triads}]
Set $q = a_{12}a_{23}a_{31}$, $t = q^{1/3}$, and $c = 1 + t + t^{-1}$. Let
$$g_i = \left(\prod_{j=1}^{3} a_{ij}\right)^{1/3}, \qquad i = 1,2,3,$$
and set $\mathbf{D}_2 = \mathrm{diag}(g_1,g_2,g_3)$, $\mathbf{D}_1 = \mathbf{D}_2^{-1}$. Consider $\mathbf{B} := \mathbf{D}_1\mathbf{A}\mathbf{D}_2 = \mathbf{D}_2^{-1}\mathbf{A}\mathbf{D}_2$, whose entries are $b_{ij} = a_{ij}(g_j/g_i)$. Since $a_{ii}=1$ for all $i$, we have $g_1=(a_{12}a_{13})^{1/3}$, $g_2=(a_{21}a_{23})^{1/3}$, $g_3=(a_{31}a_{32})^{1/3}$, from which

$$a_{12}\frac{g_2}{g_1} = t, \qquad a_{13}\frac{g_3}{g_1} = t^{-1}.$$

Hence the first row sum of $\mathbf{B}$ is $1+t+t^{-1}$ and the same holds, by an analogous computation, for the second and third rows. Moreover, by reciprocity, $b_{ij}b_{ji} = a_{ij}a_{ji} = 1$, so the column sums of $\mathbf{B}$ coincide with the row sums. Therefore, the matrix $\frac{1}{c}\mathbf{B} = \frac{1}{c}\mathbf{D}_2^{-1}\mathbf{A}\mathbf{D}_2$ is doubly stochastic. Since $\mathbf{A}$ has total support, by Theorem~\ref{SinkhornKnopp} we have the uniqueness of the doubly stochastic matrix of the form $\mathbf{D}_1\mathbf{A}\mathbf{D}_2$ with $\mathbf{D}_1,\mathbf{D}_2$ positive diagonal. So $\frac{1}{c}\mathbf{B}$ coincides with the Sinkhorn core $\mathbf{S}$ of $\mathbf{A}$, whose entries are $1/c$, $t/c$, $t^{-1}/c$, each repeated three times. Since the normalization of entropy is invariant under
a change of logarithm base, we use natural logarithms in what follows. Hence
\[
I_H(\mathbf{A})=1-\frac{H(t)}{\log 3},
\]
where
\[
H(t)=\log c-\frac{t-t^{-1}}{c}\log t.
\]
The expression \(H(t)\) is unchanged if \(t\) is replaced by \(t^{-1}\).
Therefore \(I_H(\mathbf{A})\) depends only on \(|\log t|\), and hence only
on \(|\log q|\), since \(|\log q|=3|\log t|\). By this symmetry, it is enough to consider \(t\ge 1\). A direct
differentiation gives
\[
H'(t)=
-\frac{(t^2+4t+1)\log t}{(t^2+t+1)^2}.
\]
For \(t\ge 1\), we have \(\log t\ge 0\), so \(H'(t)\le 0\), with strict
inequality for \(t>1\). Thus \(H(t)\) decreases as \(t\) moves away from
\(1\), and consequently \(I_H(\mathbf{A})=1-H(t)/\log 3\) is strictly
increasing. Since \(|\log q|=3|\log t|\), \(I_H(\mathbf{A})\) is strictly
increasing with respect to $
\left|\log{a_{12}a_{23}}{a_{31}}\right|$.
\end{proof}
}

{

\begin{proof}[Proof of Proposition \ref{prop:priority-properties}]
We shall prove all properties one by one.
\begin{itemize}
\item[Q1] It follows from the uniqueness of Sinkhorn scaling. The doubly stochastic
matrix $\mathbf S=\mathbf D_1\mathbf A\mathbf D_2$ is unique,
while the scaling matrices $\mathbf D_1$ and $\mathbf D_2$ are unique up to
a common scalar factor. More precisely, if $(\mathbf D_1,\mathbf D_2)$ is
replaced by $(c\mathbf D_1,c^{-1}\mathbf D_2)$, with $c>0$, then
\[
\sqrt{
\frac{c^{-1}d^{(2)}_i}{c d^{(1)}_i}
}
=
\frac1c
\sqrt{
\frac{d^{(2)}_i}{d^{(1)}_i}
}.
\]
Thus all components of $\mathbf w(\mathbf A)$ are multiplied by the same
positive constant. Hence the priority vector is uniquely determined up to
multiplication by a positive scalar.

\item[Q2:] It follows from the polynomial-time approximability of Sinkhorn scaling \citep{AltschulerEtAl2017}. Once the scaling matrices $\mathbf{D}_1$
and $\mathbf{D}_2$ have been computed to the prescribed accuracy, the priority
vector is obtained directly from \eqref{eq:weights}.

\item[Q3:] This follows from Proposition~\ref{prop:main}. If $\mathbf A$ is consistent,
then the diagonal of $\mathbf D_2$ is compatible with $\mathbf A$, and the
component-wise inverse of the diagonal of $\mathbf D_1$ is compatible with
$\mathbf A$. Therefore
\[
\frac{d^{(2)}_i}{d^{(2)}_j}=a_{ij},
\qquad
\frac{d^{(1)}_j}{d^{(1)}_i}=a_{ij}.
\]
It follows that
\[
\frac{w_i}{w_j}
=
\sqrt{
\frac{d^{(2)}_i/d^{(1)}_i}{d^{(2)}_j/d^{(1)}_j}
}
=
\sqrt{
\frac{d^{(2)}_i}{d^{(2)}_j}
\frac{d^{(1)}_j}{d^{(1)}_i}
}
=
\sqrt{a_{ij}^2}
=
a_{ij}.
\]

\item[Q4:] It follows from the equivariance of Sinkhorn scaling under simultaneous row and column permutations. Let $\mathbf P$ be a permutation matrix. If
$\mathbf S(\mathbf A)=\mathbf D_1\mathbf A\mathbf D_2$, then
\[
(\mathbf P\mathbf D_1\mathbf P^\top)
(\mathbf P\mathbf A\mathbf P^\top)
(\mathbf P\mathbf D_2\mathbf P^\top)
=
\mathbf P\mathbf S(\mathbf A)\mathbf P^\top.
\]
The matrix on the right-hand side is doubly stochastic. Hence, by uniqueness of the Sinkhorn-scaled matrix, $\mathbf S(\mathbf P\mathbf A\mathbf P^\top) = \mathbf P\mathbf S(\mathbf A)\mathbf P^\top$. Moreover, the diagonal entries of the corresponding scaling matrices are
permuted in the same way. Hence, $\mathbf w(\mathbf P\mathbf A\mathbf P^\top)
\propto
\mathbf P\mathbf w(\mathbf A)$.

\item[Q5:] It follows from the fact that transposition exchanges the two Sinkhorn
scaling matrices. Indeed, if
$\mathbf S(\mathbf A)=\mathbf D_1\mathbf A\mathbf D_2$, then
\[
\mathbf S(\mathbf A)^\top
=
\mathbf D_2\mathbf A^\top\mathbf D_1.
\]
Since $\mathbf S(\mathbf A)^\top$ is doubly stochastic, this is a Sinkhorn
scaling of $\mathbf A^\top$. Hence, up to a common scalar factor,
\[
\mathbf D_1(\mathbf A^\top)\propto \mathbf D_2(\mathbf A),
\qquad
\mathbf D_2(\mathbf A^\top)\propto \mathbf D_1(\mathbf A).
\]
Consequently,
\[
w_i(\mathbf A^\top)
=
\sqrt{
\frac{d^{(2)}_i(\mathbf A^\top)}
{d^{(1)}_i(\mathbf A^\top)}
}
\propto
\sqrt{
\frac{d^{(1)}_i(\mathbf A)}
{d^{(2)}_i(\mathbf A)}
}
=
\frac{1}{w_i(\mathbf A)}.
\]
Therefore $\mathbf w(\mathbf A^\top)\propto 1/\mathbf w(\mathbf A)$, where
the reciprocal is taken component-wise.

\item[Q6:] This follows from the continuous dependence of Sinkhorn scaling on the entries of a positive matrix \citep{Sinkhorn1972}. Once a normalization of the scaling factors is fixed, the diagonal entries of $\mathbf D_1$ and $\mathbf D_2$ depend continuously on the entries of $\mathbf A$. Since $\mathbf w(\mathbf A)$ is obtained from them through division and square root, it depends continuously on $\mathbf A$, up to multiplication by a positive scalar.

\item[Q7:] Finally, let $\mathbf B := \mathbf D\mathbf A\mathbf D^{-1}$ and let $\mathbf S = \mathbf D_1 \mathbf A \mathbf D_2$ be the Sinkhorn scaling of $\mathbf A$. Define
$\mathbf D_1' := \mathbf D_1\mathbf D^{-1}$ and $\mathbf D_2' := \mathbf D\mathbf D_2$, which are diagonal matrices. Then

$$\mathbf D_1' \mathbf B \mathbf D_2' = (\mathbf D_1\mathbf D^{-1})(\mathbf D\mathbf A\mathbf D^{-1})(\mathbf D\mathbf D_2) = \mathbf D_1\mathbf A\mathbf D_2 = \mathbf S,$$
so $(\mathbf D_1', \mathbf D_2')$ is a valid Sinkhorn scaling of $\mathbf B$, with the same doubly stochastic matrix $\mathbf S$ as $\mathbf A$. By Theorem \ref{SinkhornKnopp}, we have the uniqueness of the scaling matrices up to a common scalar factor, then

$$d^{(1)}_i(\mathbf B) = \frac{d^{(1)}_i(\mathbf A)}{d_i}, \qquad d^{(2)}_i(\mathbf B) = d_i\,d^{(2)}_i(\mathbf A), \qquad i=1,\dots,n.$$
Substituting into \eqref{eq:weights}, we obtain

$$w_i(\mathbf B) = \sqrt{\frac{d^{(2)}_i(\mathbf B)}{d^{(1)}_i(\mathbf B)}}
= \sqrt{\frac{d_i\,d^{(2)}_i(\mathbf A)}{d^{(1)}_i(\mathbf A)/d_i}}
= d_i\sqrt{\frac{d^{(2)}_i(\mathbf A)}{d^{(1)}_i(\mathbf A)}} = d_i\,w_i(\mathbf A),$$

for all $i$, and therefore $w(\mathbf B) \propto \mathbf D\,\mathbf w(\mathbf A)$.

\end{itemize}

Therefore, the proposed prioritization method satisfies Q1--Q7.
\end{proof}

}

\begin{proof}[Proof of Proposition \ref{prop:3x3}]
Let $\mathbf{A}\in\mathcal{A}$ and $n=3$. Let us call the row geometric mean
\[
g_i:=\left(\prod_{j=1}^3 a_{ij}\right)^{1/3}\qquad i=1,2,3,
\]
and set $\mathbf{D}_2:=\operatorname{diag}(g_1,g_2,g_3)$ and $\mathbf{D}_1:=\mathbf{D}_2^{-1}$. Consider
\[
\mathbf{B}:=\mathbf{D}_1\mathbf{A}\mathbf{D}_2
=\mathbf{D}_2^{-1}\mathbf{A}\mathbf{D}_2.
\]
Its entries are $b_{ij}={a_{ij}({g_j}/{g_i})}$. Since the diagonal entries of $\mathbf{A}$ are equal to $1$, we have
\[
g_1=(a_{12}a_{13})^{1/3}\qquad
g_2=(a_{21}a_{23})^{1/3}\qquad
g_3=(a_{31}a_{32})^{1/3}.
\]
A direct computation, and an application of reciprocity, gives
\[
a_{12}\frac{g_2}{g_1}
=
\left(a_{12}a_{23}a_{31}\right)^{1/3}
\qquad
a_{13}\frac{g_3}{g_1}
=
\left(\frac{1}{a_{12}a_{23}a_{31}}\right)^{1/3}.
\]
Hence, setting $t:=\left(a_{12}a_{23}a_{31}\right)^{1/3}$,
the first row sum of \(\mathbf{B}\) is $1+t+t^{-1}$. Similarly, the second and third row sums are also equal to $1+t+t^{-1}$. Thus, all row sums of \(\mathbf{B}\) are equal.
Moreover, by reciprocity,
\[
b_{ij}b_{ji}
=
\left(a_{ij}\frac{g_j}{g_i}\right)\left(a_{ji}\frac{g_i}{g_j}\right)
=
a_{ij}a_{ji}
=
1,
\]
so each off-diagonal pair of \(\mathbf{B}\) is reciprocal. It follows that the column sums of \(\mathbf{B}\) are equal to the row sums as well. Therefore, after division by the common sum \(c:=1+t+t^{-1}\), the matrix
\[
\frac{1}{c}\mathbf{B}
=
\frac{1}{c}\mathbf{D}_2^{-1}\mathbf{A}\mathbf{D}_2
\]
is doubly stochastic. Hence, the Sinkhorn scaling may be chosen so that the diagonal of \(\mathbf{D}_2\) is proportional to
$\mathbf{g}=(g_1,g_2,g_3)^{\top}$, which is the geometric mean method. Similar arguments show that the diagonal of $\mathbf{D}_1$ is inversely proportional to the weights obtained with the geometric mean method. By the construction above, the Sinkhorn scaling can be chosen so that
\(\mathbf{D}_2=\operatorname{diag}(g_1,g_2,g_3)\) and
\(\mathbf{D}_1=\frac{1}{c}\operatorname{diag}(g_1,g_2,g_3)^{-1}\). Therefore,
\(d_i^{(2)}=g_i\) and \(d_i^{(1)}=\frac{1}{c g_i}\). Using
\eqref{eq:weights}, we obtain
\[
w_i
=
\sqrt{\frac{d_i^{(2)}}{d_i^{(1)}}}
=
\sqrt{\frac{g_i}{1/(c g_i)}}
=
\sqrt{c}\,g_i .
\]
Hence \(\mathbf w\propto \mathbf g\). Since, for \(n=3\), the geometric
mean method and the eigenvector method yield equivalent priority vectors  \citep[p. 393]{CrawfordWilliams1985},
the priority vector obtained from \eqref{eq:weights} is equivalent to both
of them.
\end{proof}

{
\begin{proof}[Proof of Proposition \ref{prop:4x4}]
Let \(\mathbf{g}=(g_1,g_2,g_3,g_4)^\top\) be the geometric mean vector,
where \[ g_i=\left(\prod_{j=1}^{4}a_{ij}\right)^{1/4},\]
and set \(\mathbf{G}=\operatorname{diag}(g_1,g_2,g_3,g_4)\) and consider $\mathbf{B}=\mathbf{G}^{-1}\mathbf{A}\mathbf{G}.$ Since \(\mathbf{A}\) is reciprocal, \(\prod_{i=1}^{4}g_i=1\). Moreover,
\[
\prod_{j=1}^{4}b_{ij}
=
\frac{\prod_{j=1}^{4}a_{ij}\prod_{j=1}^{4}g_j}{g_i^4}
=1,
\]
so every row of \(\mathbf{B}\) has geometric mean equal to one. Any reciprocal \(4\times4\) matrix with unit row products can be written,
for suitable \(x_1,x_2,x_3,x_4>0\) satisfying
\(x_1x_2x_3x_4=1\), as
\[
\mathbf{B}=
\begin{pmatrix}
1 & x_3/x_4 & x_4/x_2 & x_2/x_3\\
x_4/x_3 & 1 & x_1/x_4 & x_3/x_1\\
x_2/x_4 & x_4/x_1 & 1 & x_1/x_2\\
x_3/x_2 & x_1/x_3 & x_2/x_1 & 1
\end{pmatrix}.
\]
Set \(p_i=x_i+x_i^{-1}\), let
\(\mathbf{p}=(p_1,p_2,p_3,p_4)^\top\), and define
\(\mathbf{P}=\operatorname{diag}(p_1,p_2,p_3,p_4)\). A direct expansion, using
\(x_1x_2x_3x_4=1\), gives
\[
(\mathbf{B}\mathbf{p})_i
=
(\mathbf{B}^{\top}\mathbf{p})_i
=
\prod_{j\ne i}p_j,
\qquad i=1,\ldots,4.
\]
Consequently, every row sum and every column sum of
\(\mathbf{P}\mathbf{B}\mathbf{P}\) is equal to $\kappa=\prod_{j=1}^{4}p_j.$
Therefore, setting
$\widehat{\mathbf{P}}=({1}/{\sqrt{\kappa}})\mathbf{P}$,
the matrix $\widehat{\mathbf{P}}\mathbf{B}\widehat{\mathbf{P}}$ is doubly stochastic. Since \(\mathbf{B}=\mathbf{G}^{-1}\mathbf{A}\mathbf{G}\), we have
$
\widehat{\mathbf{P}}\mathbf{B}\widehat{\mathbf{P}}
=
\bigl(\widehat{\mathbf{P}}\mathbf{G}^{-1}\bigr)
\mathbf{A}
\bigl(\mathbf{G}\widehat{\mathbf{P}}\bigr).
$
Thus, the Sinkhorn scaling matrices of \(\mathbf{A}\) may be chosen as $\mathbf{D}_1=\widehat{\mathbf{P}}\mathbf{G}^{-1}$, and $\mathbf{D}_2=\mathbf{G}\widehat{\mathbf{P}}$.
Hence
\[
d_i^{(1)}=\frac{p_i}{\sqrt{\kappa}\,g_i},
\qquad
d_i^{(2)}=\frac{g_ip_i}{\sqrt{\kappa}},
\]
and therefore
\[
w_i
=
\sqrt{\frac{d_i^{(2)}}{d_i^{(1)}}}
=
\sqrt{
\frac{g_ip_i/\sqrt{\kappa}}
     {p_i/(\sqrt{\kappa}\,g_i)}
}
=
g_i.
\]
Thus, the weight vector obtained from the Sinkhorn factors is equivalent
to the geometric mean vector.
\end{proof}
}

\begin{proof}[Proof of Proposition \ref{prop:SinkhornOTS}]
Sinkhorn scaling produces diagonal matrices $\mathbf{D}_1 = \mathrm{diag}(u_1,\dots,u_n)$ and $\mathbf{D}_2 = \mathrm{diag}(v_1,\dots,v_n)$ such that $\mathbf{S} = \mathbf{D}_1 \mathbf{A} \mathbf{D}_2$ satisfies the marginal constraints and is therefore doubly stochastic. Define the kernel matrix
\[
\mathbf{K} = \exp(-\mathbf{C}/\alpha) = \exp(\log \mathbf{A}) = \mathbf{A}.
\]
It is well-known \citep{Solomon2018} that the solution of the entropy-regularized transport problem with kernel $\mathbf{K}$ has the form
\[
\mathbf{X}_\alpha = \mathrm{diag}(u_1,\dots,u_n) \, \mathbf{K} \, \mathrm{diag}(v_1,\dots,v_n).
\]
Substituting $\mathbf{K} = \mathbf{A}$ gives
\[
\mathbf{X}_\alpha = \mathrm{diag}(u_1,\dots,u_n) \, \mathbf{A} \, \mathrm{diag}(v_1,\dots,v_n) = \mathbf{D}_1 \mathbf{A} \mathbf{D}_2 = \mathbf{S}.
\]

Hence, the Sinkhorn scaling procedure and the solution of the entropy-regularized transport problem are equivalent. Moreover, if $\mathbf{A}$ is consistent, Proposition~\ref{prop:main} implies
\[
\mathbf{S} = \mathbf{D}_1 \mathbf{A} \mathbf{D}_2 = \frac{1}{n} \mathbf{1} \mathbf{1}^\top,
\]
so all preferential information is absorbed into $\mathbf{D}_2$
(or equivalently into the componentwise inverse of the diagonal of
$\mathbf{D}_1$, up to a common factor), and $\mathbf{S}$ is uniform.
\end{proof}

\end{document}